\documentclass[11pt,english]{article}
\usepackage[T1]{fontenc}
\usepackage[latin9]{inputenc}
\usepackage{amsmath}
\usepackage{amssymb}

\makeatletter

\providecommand{\tabularnewline}{\\}

\@ifundefined{definecolor}
 {\usepackage{color}}{}

\@ifundefined{definecolor}
 {\usepackage{color}}{}

\@ifundefined{definecolor}
 {\usepackage{color}}{}

\usepackage{amsfonts}

\usepackage{amsfonts}

\usepackage{a4wide}

\usepackage{theorem}

\providecommand{\U}[1]{\protect\rule{.1in}{.1in}}
\makeatletter

\@ifundefined{definecolor}

\newtheorem{theorem}{Theorem}

\newtheorem{corollary}[theorem]{Corollary}

\newtheorem{lemma}[theorem]{Lemma}

\newtheorem{proposition}[theorem]{Proposition}
{\theorembodyfont{\rmfamily}

\newtheorem{example}[theorem]{Example}
\newtheorem{remark}{Remark}
}

\newenvironment{proof}[1][Proof]{\noindent\textbf{#1.} }{\ \rule{0.5em}{0.5em}}
\makeatother

\usepackage{babel}
\makeatother

\begin{document}

\title{Linear Monotone Subspaces \\
of Locally Convex Spaces}

\author{M.D. Voisei and C. Z\u{a}linescu}

\date{{}}

\maketitle
\begin{abstract}
The main focus of this paper is to study multi-valued linear monotone
operators in the contexts of locally convex spaces via the use of
their Fitzpatrick and Penot functions. Notions such as maximal monotonicity,
uniqueness, negative-infimum, and (dual-) representability are studied
and criteria are provided.
\end{abstract}

\section{Preliminaries}

Motivated by the facts that, besides the subdifferentials, single-valued
linear maximal monotone operators enjoy a set of stronger monotonicity
properties and belong to most of the special classes of maximal monotone
operators introduced in the non-reflexive Banach space settings the
linear monotone subspaces of a non-reflexive Banach space made the
object of extensive studies in \cite{Bauschke-Borwein 99,Bauschke-Simons 99,Phelps-Simons 98,Bauschke Wang Yao 2008}.

It is also hoped that a linear maximal monotone would be a part of
a counter-example to the celebrated Rockafellar Conjecture (see e.g.
\cite{Borwein06-via}), that is why, a thorough study of linear monotone
operators could help accomplish that goal.

The vast majority of the results concerning monotone operators are
stated in the context of Banach spaces with a few exceptions (see
e.g. \cite{Browder68-fixed point,Fitzpatrick:88}). In the context
of non-reflexive Banach spaces it is natural to use topologies compatible
with the simple duality $(X\times X^{*},X^{*}\times X)$ instead of
the strong topologies. This leads to the context of locally convex
spaces under which, fortunately, many of the results stated in Banach
spaces for monotone operators still hold.

Our goal in this paper is twofold: - to extend and add to the known
results concerning maximal monotone multi-valued operators in the
context of locally convex spaces and - to provide general criteria
for various classes of linear monotone operators while performing
a comparison of the main types of linear maximal monotone operators
through examples and counter-examples.

It is interesting to mention that in a locally convex space context,
currently, there are two characterization of maximal monotonicity
of a general operator:

\begin{itemize}
\item (Fitzpatrick \cite[Theorem 3.8]{Fitzpatrick:88}) Let $X$ be a locally
convex space and $T:X\rightrightarrows X^{*}$ be monotone. Then $T$
is maximal monotone iff $\varphi_{T}(x,x^{*})=\sup\{x^{*}(y)+y^{*}(x-y)\mid(y,y^{*})\in\mathrm{gph}\, T\}>x^{*}(x)$,
for every $(x,x^{*})\in X\times X^{*}\setminus\mathrm{gph}\, T$.
\item (Voisei \& Z\u{a}linescu \cite[Theorem 1]{VZ mmcs 2008}) Let $X$
be a locally convex space and $T:X\rightrightarrows X^{*}$. Then
$T$ is maximal monotone iff $T$ is representable and $T$ is NI
in $X\times X^{*}$.
\end{itemize}
Here, we provide another characterization for the maximality of a
monotone operator in terms of uniqueness, namely

\begin{itemize}
\item Let $X$ be a locally convex space and $T:X\rightrightarrows X^{*}$
be monotone. Then $T$ is maximal monotone iff $T$ is unique and
dual-representable (see Theorem \ref{dr-unic} below).
\end{itemize}
The plan of the paper is as follows. In the second section the main
notions and notation together with the most representative results
concerning these notions are presented. Section 3 studies general
monotone operator properties such as uniqueness, NI type, graph-convexity,
(dual-)representability, and maximal monotonicity. Section 4 deals
with skew linear monotone operators. In Section 5 the main types of
linear monotone subsets are studied and criteria are provided.

\section{Main notions and notations}

For a locally convex space $(E,\mu)$ and $A\subset E$, we denote
by {}``$\operatorname*{conv}A$'' the \emph{convex hull} of $A$,
{}{}{}``$\operatorname*{aff}A$'' the \emph{affine hull} of $A$,
{}{}{}``$\operatorname*{lin}A$'' the \emph{linear hull} of $A$,
{}{}{}``$\operatorname*{cl}_{\mu}A$'' the $\mu-$\emph{closure}
of $A$, {}``$A^{i}$'' the \emph{algebraic interior} of $A$, {}``$^{i}A$''
the \emph{relative algebraic interior} of $A$ with respect to $\operatorname*{aff}A$,
while $^{ic}A:={}^{i}A$ if $\operatorname*{aff}A$ is $\mu-$closed
and $^{ic}A:=\emptyset$ otherwise, is the relative algebraic interior
of $A$ with respect to $\operatorname*{cl}_{\mu}(\operatorname*{aff}A)$.
In the sequel when the topology is implicitly understood we avoid
the use of the $\mu-$notation. A subset $A$ is a \emph{cone} if
$\mathbb{R}_{+}A=A$ while $A$ is a \emph{double-cone} if $\mathbb{R}A=A$.

For $f,g:E\rightarrow\overline{\mathbb{R}}:=\mathbb{R}\cup\{-\infty,+\infty\}$
we set $[f\leq g]:=\{x\in E\mid f(x)\leq g(x)\}$; the sets $[f=g]$,
$[f<g]$ and $[f>g]$ are defined similarly.

Throughout this paper, if not otherwise explicitly mentioned, $(X,\tau)$
is a non trivial (that is, $X\neq\{0\}$) separated locally convex
space, $X^{\ast}$ is its topological dual endowed with the weak-star
topology $w^{\ast}$, the topological dual of $(X^{\ast},w^{\ast})$
is identified with $X$, and the weak topology on $X$ is denoted
by $w$. The \emph{duality product} of $X\times X^{\ast}$ is denoted
by $\left\langle x,x^{\ast}\right\rangle :=x^{\ast}(x)=:c(x,x^{\ast})$
for $x\in X$, $x^{\ast}\in X^{\ast}$.

To an operator (or multi\-function) $T:X\rightrightarrows X^{\ast}$
we associate its \emph{graph}: $\operatorname*{gph}T=\{(x,x^{\ast})\in X\times X^{\ast}\mid x^{\ast}\in T(x)\}$,
\emph{inverse:} $T^{-1}:X^{*}\rightrightarrows X$, $\operatorname*{gph}T^{-1}=\{(x^{*},x)\mid(x,x^{\ast})\in\operatorname*{gph}T\}$,
\emph{domain}: $\operatorname*{dom}T:=\{x\in X\mid T(x)\neq\emptyset\}=\Pr\nolimits _{X}(T)$,
and \emph{range}: $\operatorname*{Im}T:=\{x^{*}\in X^{*}\mid x^{*}\in T(x)\ \mathrm{for\ some}\ x\in X\}=\Pr\nolimits _{X^{*}}(T)$.
Here $\Pr_{X}$ and $\Pr_{X^{\ast}}$ are the projections of $X\times X^{*}$
onto $X$ and $X^{\ast}$, respectively. When no confusion can occur,
$T$ will be identified with $\operatorname*{gph}T$.

\strut

On $X$, we consider the following classes of functions and operators:

\begin{description}
\item [{$\Lambda(X)$}] the class formed by proper convex functions $f:X\rightarrow\overline{\mathbb{R}}$.
Recall that $f$ is \emph{proper} if $\mathrm{dom}\: f:=\{x\in X\mid f(x)<\infty\}$
is nonempty and $f$ does not take the value $-\infty$,
\item [{$\Gamma_{\tau}(X)$}] the class of functions $f\in\Lambda(X)$
that are $\tau$--lower semi\emph{\-}continuous (\emph{$\tau$--}lsc
for short); when the topology is implicitly understood we use the
notation $\Gamma(X)$,
\item [{$\mathcal{M}(X)$}] the class of monotone operators $T:X\rightrightarrows X^{*}$.
Recall that $T:X\rightrightarrows X^{*}$ is \emph{monotone} if $\left\langle x_{1}-x_{2},x_{1}^{\ast}-x_{2}^{\ast}\right\rangle \ge0$,
for all $x_{1}^{\ast}\in Tx_{1}$, $x_{2}^{\ast}\in Tx_{2}$.
\item [{$\mathfrak{M}(X)$}] the class of maximal monotone operators $T:X\rightrightarrows X^{*}$.
The maximality is understood in the sense of graph inclusion as subsets
of $X\times X^{*}$.
\end{description}
Notions associated to a proper function $f:X\rightarrow\overline{\mathbb{R}}$:

\begin{description}
\item [{$\operatorname*{epi}f:=\{(x,t)\in X\times\mathbb{R}\mid f(x)\leq t\}$}] is
the \emph{epigraph} of $f$,
\item [{$\operatorname*{conv}f:X\rightarrow\overline{\mathbb{R}}$,}] the
\emph{convex hull} of $f$, is the greatest convex function majorized
by $f$, $(\operatorname*{conv}f)(x):=\inf\{t\in\mathbb{R}\mid(x,t)\in\operatorname*{conv}(\operatorname*{epi}f)\}$
for $x\in X$,
\item [{$\operatorname*{cl}_{\tau}\operatorname*{conv}f:X\rightarrow\overline{\mathbb{R}}$,}] the
\emph{$\tau-$lsc convex hull} of $f$, is the greatest \emph{$\tau$--}lsc
convex function majorized by $f$, $(\operatorname*{cl}_{\tau}\operatorname*{conv}f)(x):=\inf\{t\in\mathbb{R}\mid(x,t)\in\operatorname*{cl}_{\tau}(\operatorname*{conv}\operatorname*{epi}f)\}$
for $x\in X$,
\item [{$f^{\ast}:X^{\ast}\rightarrow\overline{\mathbb{R}}$}] is the \emph{convex
conjugate} of $f:X\rightarrow\overline{\mathbb{R}}$ with respect
to the dual system $(X,X^{\ast})$, $f^{\ast}(x^{\ast}):=\sup\{\left\langle x,x^{\ast}\right\rangle -f(x)\mid x\in X\}$
for $x^{\ast}\in X^{\ast}$.
\item [{$\partial f(x)$}] is the \emph{subdifferential} of $f$ at $x\in X$;
$\partial f(x):=\{x^{\ast}\in X^{\ast}\mid\left\langle x^{\prime}-x,x^{\ast}\right\rangle +f(x)\leq f(x^{\prime}),\ \forall x^{\prime}\in X\}$
for $x\in X$ (it is clear that $\partial f(x):=\emptyset$ for $x\not\in\operatorname*{dom}f$).
Recall that $N_{C}=\partial\iota_{C}$ is the \emph{normal cone} of
$C$, where $\iota_{C}$ is the \emph{indicator} \emph{function} of
$C\subset X$ defined by $\iota_{C}(x):=0$ for $x\in C$ and $\iota_{C}(x):=\infty$
for $x\in X\setminus C$.
\end{description}
Let $Z:=X\times X^{\ast}$. It is known that $(Z,\tau\times w^{\ast})^{\ast}=Z$
via the coupling \[
z\cdot z^{\prime}:=\left\langle x,x^{\prime\ast}\right\rangle +\left\langle x^{\prime},x^{\ast}\right\rangle ,\quad\text{for }z=(x,x^{\ast}),\ z^{\prime}=(x^{\prime},x^{\prime\ast})\in Z.\]
 For a proper function $f:Z\rightarrow\overline{\mathbb{R}}$ all
the above notions are defined similarly. In addition, with respect
to the natural dual system $(Z,Z)$ induced by the previous coupling,
the conjugate of $f$ is given by\[
f^{\square}:Z\rightarrow\overline{\mathbb{R}},\quad f^{\square}(z)=\sup\{z\cdot z^{\prime}-f(z^{\prime})\mid z^{\prime}\in Z\},\]
 and by the biconjugate formula, $f^{\square\square}=\mathrm{cl}_{\tau\times w^{\ast}}\mathrm{\, conv}\, f$
whenever $f^{\square}$ (or $\mathrm{cl}_{\tau\times w^{\ast}}\mathrm{\, conv}\, f$)
is proper.

To a multi\-function $T:X\rightrightarrows X^{\ast}$ we associate
the following functions: $c_{T}:Z\rightarrow\overline{\mathbb{R}}$,
$c_{T}:=c+\iota_{T}$, $\psi_{T}:Z\rightarrow\overline{\mathbb{R}}$,
$\psi_{T}:=\operatorname*{cl}\,\!_{\tau\times w^{\ast}}\operatorname*{conv}c_{T}$
is the \emph{Penot function} of $T$, $\varphi_{T}:Z\rightarrow\overline{\mathbb{R}}$,
$\varphi_{T}:=c_{T}^{\square}=\psi_{T}^{\square}$ is the \emph{Fitzpatrick
function} of $T$.

From the definition of $\varphi_{T}$ one has (as observed in \cite[Proposition 2]{Voisei:06b})
\begin{equation}
T\subset(\operatorname*{dom}T\times X^{\ast})\cup(X\times\operatorname{Im}\mathrm{\,}T)\subset[\varphi_{T}\geq c].\label{r1}\end{equation}
 Moreover, as observed in several places (see e.g.\ \cite{martinez-legaz svaiter 05,Penot:04,Voisei:amt 06,Voisei:06b,Voisei tscr:06,VZ mmcs 2008}),
\begin{gather}
T\in\mathcal{M}(X)\Longleftrightarrow\operatorname*{conv}c_{T}\geq c\Longleftrightarrow\psi_{T}\ge c\Longleftrightarrow T\subset[\varphi_{T}=c]\Longleftrightarrow T\subset[\varphi_{T}\leq c],\label{r2}\\
T\in\mathcal{M}(X)\Rightarrow T\subset[\psi_{T}=c]\subset[\varphi_{T}=c],\label{r3}\end{gather}
 and \begin{equation}
T\in\mathfrak{M}(X)\Longleftrightarrow T=[\varphi_{T}\leq c]\Longleftrightarrow\big[\varphi_{T}\geq c\text{ and }T=[\varphi_{T}=c]\big].\label{r4}\end{equation}

To the equivalences in (\ref{r4}) we add the following\begin{equation}
T\in\mathfrak{M}(X)\Leftrightarrow T\in\mathcal{M}(X),\ [\varphi_{T}\leq c]\subset T.\label{r4b}\end{equation}

While, from the last equivalence in (\ref{r4}), the direct implication
is obvious, for the converse implication we have, by (\ref{r1}),
that $[\varphi_{T}\leq c]\subset T\subset[\varphi_{T}\geq c]$; whence
$\varphi_{T}\ge c$ and $[\varphi_{T}=c]\subset T$. Since $T$ is
monotone, by (\ref{r3}), we get $T=[\varphi_{T}=c]$ and so, from
(\ref{r4}), $T\in\mathfrak{M}(X)$.

Note that condition $T\in\mathcal{M}(X)$ is not superfluous in (\ref{r4b}).
Take for example the non-monotone $T=\{(x,y)\in\mathbb{R}^{2}\mid xy\ge0\}$
for which $[\varphi_{T}\le c]=\{(0,0)\}\subset T$.

The preceding relations suggest the introduction of the following
classes of functions: \begin{gather*}
\mathcal{F}:=\mathcal{F}(Z):=\{f\in\Lambda(Z)\mid f\geq c\},\\
\mathcal{R}:=\mathcal{R}(Z):=\Gamma_{\tau\times w^{\ast}}(Z)\cap\mathcal{F}(Z),\\
\mathcal{D}:=\mathcal{D}(Z):=\{f\in\mathcal{R}(Z)\mid f^{\square}\geq c\}.\end{gather*}

It is known that $[f=c]\in\mathcal{M}(X)$, for every $f\in\mathcal{F}(Z)$
(see e.g.\ \cite{Penot:04}).

We consider the following classes of multi\-functions $T:X\rightrightarrows X^{\ast}$:

\begin{itemize}
\item $T$ is \emph{representable} in $Z$ if $T=[f=c]$, for some $f\in\mathcal{R}$;
in this case $f$ is called a \emph{representative} of $T$. We denote
by $\mathcal{R}_{T}$ the class of representatives of $T$.
\item $T$ is \emph{dual-representable} if $T=[f=c]$, for some $f\in\mathcal{D}$;
in this case $f$ is called a \emph{d--representative} of $T$. We
denote by $\mathcal{D}_{T}$ the class of d-representatives of $T$.
\item $T$ is of \emph{negative infimum type in $Z$} (NI for short) if
$\varphi_{T}\geq c$ in $Z$. Note that in the case $X$ is a Banach
space, this notion is a weaker form relative to $X^{*}\times X^{**}$
of the original version introduced by Simons \cite{Simons carte}
in the sense that $T$ is NI in the sense of Simons means that $T$
is maximal monotone in $X\times X^{*}$ and $T^{-1}$ is NI in $X^{*}\times X^{**}$
in the present sense. However, if $T^{-1}$ is NI in $X^{*}\times X^{**}$
in the present sense, then $T$ is NI in $X\times X^{*}$. In this
form this notion was first considered in \cite[Remark 3.5]{Voisei tscr:06}.
\item $T$ is \emph{unique} in $Z$ if $T$ is monotone and admits a unique
maximal monotone extension in $Z$. In the context of Banach spaces
in \cite{martinez-legaz svaiter 05} the preceding notion was considered
under the name of pre-maximal monotone operator. Previously, the uniqueness
notion was used in \cite{Gossez72,Simons:96,Bauschke-Borwein 99};
$T$ is unique in their sense of iff $T^{-1}$ (as a subset of $X^{*}\times X^{**}$)
is unique in $X^{*}\times X^{**}$ in the present sense.
\end{itemize}
Fitzpatrick proved in \cite[Theorem 2.4]{Fitzpatrick:88} that \begin{equation}
f\in\mathcal{F}\Longrightarrow[f=c]\subset[f^{\square}=c],\label{r5}\end{equation}
 from which, it follows that \begin{equation}
[f=c]=[f^{\square}=c]\qquad\forall f\in\mathcal{D}.\label{r6}\end{equation}
 As observed in \cite[Remark 3.6]{Voisei tscr:06} (see also \cite{VZ mmcs 2008}),
if $f\in\mathcal{R}_{T}$, that is, $T$ is representable with $f$
a representative of $T$, then we have \begin{equation}
\varphi_{T}\leq f\leq\psi_{T},\quad\varphi_{T}\leq f^{\square}\leq\psi_{T}.\label{r7}\end{equation}
 Hence, if $T\in\mathfrak{M}(X)$ and $f\in\mathcal{R}_{T}$ then
$f\in\mathcal{D}_{T}$. Moreover, \begin{equation}
\big(f\in\mathcal{R},\ \ T\subset[f=c]\big)\Longrightarrow\big(f\le\psi_{T},\ \ [\psi_{T}=c]\subset[f=c]\big).\label{r8}\end{equation}
 Indeed, $T\subset[f=c]$ imply $f\le c_{T}$, whence, because $f\in\mathcal{R}$,
$f\le\psi_{T}$; therefore $[\psi_{T}=c]\subset[f=c]$.

In particular, (\ref{r8}) shows that $[\psi_{T}=c]$ is the smallest
representable extension of $T\in\mathcal{M}(X)$ in the sense of graph
inclusion and so $[\psi_{T}=c]=\bigcap\{M\mid T\subset M,\ M\ \mathrm{representable}\}$.

\begin{remark} \label{types repres} For a Banach space $X$ with
topological dual $X^{*}$, both endowed with their strong topology
denoted by {}``$s$'', there have been considered three different
notions of representability of a monotone operator $T:X\rightrightarrows X^{*}$:

\begin{itemize}
\item the one above (first introduced in \cite{Voisei:amt 06}) that will
be used throughout this article,
\item Martinez-Legaz--Svaiter (MLS) representability. In \cite{martinez-legaz svaiter 05}
one says that $T$ is representable if there exists an $s\times s-$lsc
convex function $h:X\times X^{*}\to\overline{\mathbb{R}}$ such that
$h\ge c$ and $T=[h=c]$. Note that these two types of representability
coincide if $X$ is reflexive but they are different in the non-reflexive
case. Indeed, consider $E$ a non-reflexive Banach space, $X:=E^{*}$
and $S:=\{0\}\times E\subset X\times X^{*}$ (here $E$ is identified
with its image by the canonical injection of $E$ into $E^{**}=X^{*}$).
Clearly $S$ is a skew strongly closed linear space in $X\times X^{*}$,
and $S=[\iota_{S}=c]$ which makes $S$ MLS-representable. However
$S$ is not representable in our sense because its Penot function
in $X\times X^{*}$, namely, $\psi_{S}=\operatorname*{cl}_{s\times w^{*}}\iota_{S}=\iota_{\{0\}\times X^{*}}\neq\operatorname*{cl}_{s\times s}\iota_{S}=\iota_{\{0\}\times E}$
and so $[\psi_{S}=c]=\{0\}\times X^{*}=\operatorname*{cl}_{s\times w^{*}}S\supsetneqq S$.
\item Borwein representability involves no topology (see \cite[Section 2.2]{Borwein06-via})
or the strong topology on $X\times X^{*}$ (see \cite[page 3918]{Borwein07-gresita})
and the equality $T=[h=c]$ in the MLS-representability is replaced
by the inclusion $T\subset[h=c]$ for an $s\times s-$lsc convex function
$h:X\times X^{*}\to\overline{\mathbb{R}}$ with $h\ge c$. While it
is clear that Borwein representability of a monotone operator is different
from the other two, in the case of a maximal monotone operator Borwein's
representability coincides with the MLS-representability.
\end{itemize}
\end{remark}

\section{Types of monotone operators}

The next characterizations of representability and maximality were
stated in the context of Banach spaces but their arguments work in
a locally convex settings (see also \cite[Theorem 1]{VZ mmcs 2008}).

\begin{theorem} \label{caracter-max} \emph{(\cite[Ths.\ 2.2, 2.3]{Voisei:amt 06})}

\emph{(i)} $T$ is representable iff $T\in\mathcal{M}(X)$ and $T=[\psi_{T}=c]$,
that is, $\psi_{T}$ is a representative of $T$ (or $\psi_{T}\in\mathcal{R}_{T}$),

$\emph{(ii)}$ $T$ is maximal monotone iff $T$ is representable
and $T$ is of negative infimum type, that is, $\varphi_{T}$ is a
(d--)representative of $T$ (or $\varphi_{T}\in\mathcal{R}_{T}$,
or $\varphi_{T}\in\mathcal{D}_{T}$). \end{theorem}

It is important to notice that, in the context of non-reflexive Banach
spaces, the previous characterization of maximality fails if our representability
is replaced by the MLS-representability. Indeed, let $X$ be a non-reflexive
Banach space and $S=\{0\}\times X\subsetneq\{0\}\times X^{**}\subset Z^{*}$.
As previously seen in Remark \ref{types repres}, $S$ is MLS-representable
since $\iota_{S}$ is strongly lsc convex, $\iota_{S}\ge c$, and
$S=[\iota_{S}=c]$, while $S$ is NI in $Z^{*}$ because $\varphi_{S}(x^{*},x^{**})=\infty$
for $x^{**}\neq0$ and $\varphi_{S}(x^{*},0)=0$ for $x^{*}\in X^{*}$,
and clearly $S$ is not maximal monotone in $Z^{*}$.

\strut

As in \cite{martinez-legaz svaiter 05} (for $X$ a Banach space and
with a different notation), for a subset $A$ of $Z$ we set\[
A^{+}:=[\varphi_{A}\leq c]=\{z\in Z\mid c(z-w)\geq0,\ \forall w\in A\},\]
 the set of all $z\in Z$ that are monotonically related to (m.r.t.
for short) $A$ and $A^{++}:=(A^{+})^{+}$. Note that $\emptyset^{+}=Z$,
$Z^{+}=\emptyset$, $A\subset B\subset Z$ implies $B^{+}\subset A^{+}$
and $\left(\cup_{i\in I}A_{i}\right)^{+}=\cap_{i\in I}A_{i}^{+}$
for any family $(A_{i})_{i\in I}$ of subsets of $Z$. Moreover, for
$A\subset Z$ one has:\begin{equation}
A\in\mathcal{M}(X)\Leftrightarrow A\subset A^{+},\quad A\in\mathfrak{M}(X)\Leftrightarrow A=A^{+}.\label{m1}\end{equation}
 For $T\in\mathcal{M}(X)$ we have \begin{equation}
T^{+}={\textstyle \bigcup}\{M\mid M\in\mathfrak{M}(X),\ T\subset M\},\label{m2}\end{equation}
 and so\begin{equation}
T^{++}={\textstyle \bigcap}\{M\mid M\in\mathfrak{M}(X),\ T\subset M\}.\label{m3}\end{equation}
 In particular, for $T\in\mathcal{M}(X)$ we have that $T^{++}\in\mathcal{M}(X)$
and $T\subset T^{++}\subset T^{+}$. Moreover, from (\ref{m2}) it
is easily noticed that $T\in\mathcal{M}(X)$ is unique iff $T^{+}$
is maximal monotone.

Note that the above mentioned properties of $T^{+}$ can also be found
in \cite{martinez-legaz svaiter 05}.

\begin{lemma} \label{ineq A+}\emph{(i)} For every $A\subset Z$
we have $\psi_{A}\geq\varphi_{A^{+}}$ in $Z$.

\emph{(ii)} If $T\in\mathcal{M}(X)$ then $T^{+}$ is NI, that is
$\varphi_{T^{+}}\geq c$ in $Z$. \end{lemma}

\begin{proof} (i) Notice that $A\subset A^{++}=[\varphi_{A^{+}}\le c]$
for every $A\subset Z$. Hence $\varphi_{A^{+}}\le c_{A}$ and so
$\varphi_{A^{+}}\le\psi_{A}$ in $Z$ since $\varphi_{A^{+}}$ is
convex and $w\times w^{*}-$lsc.

(ii) Let $M\subset Z$ be a maximal monotone extension of $T$. Then
$M\subset T^{+}$, and so $\varphi_{T^{+}}\ge\varphi_{M}\ge c$ in
$Z$. \end{proof}

\strut

For uniqueness one has the following characterizations (see also Proposition
\ref{dc unic} below for double-cones); these characterizations can
also be found in \cite[Proposition 36]{martinez-legaz svaiter 05}),
\cite[Theorem 19]{Simons:96}, and \cite[Fact 2.6]{Bauschke-Borwein 99}.

\begin{proposition} \label{caracter-unic} Let $T\in\mathcal{M}(X)$.
TFAE:

\emph{(i)} $T$ is unique,

\emph{(ii)} $T^{+}$ is monotone,

\emph{(iii)} $T^{+}$ is maximal monotone,

\emph{(iv)} $T^{+}=T^{++}$,

\emph{(v)} $T^{++}$ is maximal monotone.

In this case the unique maximal monotone extension of $T$ is $T^{+}=T^{++}$.
\end{proposition}

\begin{proof} Note that (i) $\Rightarrow$ (iii) is true because
of (\ref{m2}), (iii) $\Rightarrow$ (ii) is obvious, while (iii)
$\Leftrightarrow$ (iv) and (iv) $\Rightarrow$ (v) follow from the
second part of (\ref{m1}). The implication (ii) $\Rightarrow$ (iv)
holds since whenever $T^{+}$ is monotone we have, by the first part
in (\ref{m1}), that $T^{+}\subset T^{++}$; the reversed inclusion
being true because $T$ is monotone. To complete the proof it suffices
to show (v) $\Rightarrow$ (i). To this end we see from (\ref{m3})
that if $T^{++}$ is maximal monotone then $T^{++}$ is the unique
maximal monotone extension of $T$. \end{proof}

\strut

As it would be expected from Theorem \ref{caracter-max}, the following
characterization theorem shows that the NI condition for $T\in\mathcal{M}(X)$
is actually equivalent to the maximality of $[\psi_{T}=c]$ which
is the minimal representable operator that contains $T$.

\begin{proposition} \label{caracter-NI}Let $T\in\mathcal{M}(X)$.
Then

\emph{(i)} $\varphi_{T}=\varphi_{[\psi_{T}=c]}$ and $\psi_{T}=\psi_{[\psi_{T}=c]}$;
in particular $T^{+}=[\psi_{T}=c]^{+}$.

\emph{(ii)} $T$ is NI iff $[\psi_{T}=c]$ is NI iff $[\psi_{T}=c]$
is maximal monotone iff $T$ has a unique representable extension.

\emph{(iii)} If $T$ is NI then $T$ is unique and \[
T^{+}=[\psi_{T}=c]=[\varphi_{T}=c]=[\varphi_{T}\leq c]\]
 is the unique representable extension and the unique maximal monotone
extension of $T$. \end{proposition}

\begin{proof} Set $R:=[\psi_{T}=c]$ the smallest representable extension
of $T\in\mathcal{M}(X)$.

(i) From (\ref{r3}), we see that $T\subset R$, and so $\varphi_{T}\leq\varphi_{R}$.
Conversely, by the Fenchel inequality, for all $z,w\in Z$\[
\varphi_{T}(z)\geq z\cdot w-\psi_{T}(w).\]
 Pass to supremum over $w\in R$ to find $\varphi_{T}(z)\geq\varphi_{R}(z)$
for every $z\in Z$. The other relation is obtained by conjugation.

(ii) According to (i) and from Theorem \ref{caracter-max} (ii) applied
for $R$ we get that $R$ is maximal monotone iff $R$ is NI iff $T$
is NI.

Assume now that $R$ is maximal monotone and $h\in\mathcal{R}$ is
such that $T\subset[h=c]$. By (\ref{r8}) we have $R\subset[h=c]$.
Because $R$ is maximal monotone and $[h=c]$ is monotone we obtain
that $R=[h=c]$, and so $R$ is the unique representable extension
of $T$. Conversely, assume that $T$ has a unique representable extension.
Since any maximal monotone extension of $T$ is also representable
it follows that $R$ is maximal monotone and the unique representable
extension of $T$.

(iii) Assume that $T$ is NI. Hence $T^{+}=[\varphi_{T}\le c]=[\varphi_{T}=c]$.
By (ii) $R$ is the unique representable extension of $T$; in particular
$R$ is the unique maximal monotone extension of $T$, and so, according
to Proposition \ref{caracter-unic}, $R=T^{+}$. \end{proof}

\strut

Recently in \cite[Theorem 4.2]{Bauschke Wang Yao 2008} it is proved,
in the context of a Banach space, that a maximal monotone operator
with a convex graph is in fact affine. This result holds in locally
convex spaces as a consequence of the broader fact that the affine
hull of a monotone convex multifunction remains monotone.

\begin{proposition} \label{aff-conv}Let $T\subset X\times X^{*}$.
TFAE:

\emph{(i)} $c_{T}$ is convex,

\emph{(ii)} $T$ is monotone and convex.

In this case $\operatorname*{aff}T$ is monotone and $T\subset\operatorname*{aff}T\subset T^{+}$.
In particular, if $T$ is maximal monotone and convex then $T$ is
affine. \end{proposition}

\begin{proof} (i) $\Rightarrow$ (ii) Since $c_{T}$ is convex so
is $\operatorname*{dom}c_{T}=T$, while $c_{T}\ge c$ and $T=[c_{T}=c]$
show that $T$ is monotone.

(ii) $\Rightarrow$ (i) For all $\lambda\in[0,1]$ and $z,z^{\prime}\in T=\operatorname*{dom}c_{T}$
we have $\lambda z+(1-\lambda)z^{\prime}\in T$ since $T$ is convex
and\[
c_{T}(\lambda z+(1-\lambda)z^{\prime})-\lambda c_{T}(z)-(1-\lambda)c_{T}(z^{\prime})=-\lambda(1-\lambda)c(z-z^{\prime})\le0,\]
 because $T$ is monotone. Hence $c_{T}$ is convex.

Assume that $T$ is nonempty, convex and monotone. Because $(T-z)^{+}=T^{+}-z$
for every $z\in Z$, we may (and do) assume that $0\in T$. In this
case $\operatorname*{aff}T=\mathbb{R}_{+}T-\mathbb{R}_{+}T$ since
$\mathbb{R}_{+}T-\mathbb{R}_{+}T$ is linear.

Let $z,z^{\prime}\in T$, $\alpha,\alpha^{\prime}\in\mathbb{R}_{+}$.
For $\alpha+\alpha^{\prime}\neq0$ we have $\frac{\alpha}{\alpha+\alpha^{\prime}}z,\frac{\alpha^{\prime}}{\alpha+\alpha^{\prime}}z^{\prime}\in T$
because $T$ is convex and $0\in T$. Since $T$ is monotone we get\[
c(\alpha z-\alpha^{\prime}z^{\prime})=(\alpha+\alpha^{\prime})^{2}c\big(\frac{\alpha}{\alpha+\alpha^{\prime}}z-\frac{\alpha^{\prime}}{\alpha+\alpha^{\prime}}z^{\prime}\big)\ge0,\]
 and this implies that $\operatorname*{aff}T$ is monotone.

To prove that $\operatorname*{aff}T\subset T^{+}$ first note that
$[1,\infty)T^{+}\subset T^{+}$. Indeed for $z\in T^{+}$, $t\ge1$,
and an arbitrary $u\in T$ we have $t^{-1}u\in T$ since $T$ is convex
and $0\in T$. Consequently, the monotonicity of $T$ provides us
with\[
c(tz-u)=t^{2}c(z-t^{-1}u)\ge0,\]
 that is, $tz\in T^{+}$ for all $z\in T^{+}$ and $t\ge1$.

Recall that (see (\ref{m1})) $T$ monotone is equivalent to $T\subset T^{+}$,
from which $[1,\infty)T\subset[1,\infty)T^{+}\subset T^{+}$. Since
by the convexity of $T$, $[0,1]T=T$, we obtain $\mathbb{R}_{+}T\subset T^{+}$.

It is time to show that $\operatorname*{aff}T=\mathbb{R}_{+}T-\mathbb{R}_{+}T\subset T^{+}$.
Indeed, let $\alpha,\alpha^{\prime}\ge0$ and $z,z^{\prime}\in T$.
Again, for arbitrary $u\in T$ we have\[
c(\alpha z-\alpha^{\prime}z^{\prime}-u)=(1+\alpha^{\prime})^{2}c\big(\frac{\alpha^{\prime}}{1+\alpha^{\prime}}z^{\prime}+\frac{1}{1+\alpha^{\prime}}u-\frac{\alpha}{1+\alpha^{\prime}}z\big)\ge0,\]
 because $\frac{\alpha^{\prime}}{1+\alpha^{\prime}}z^{\prime}+\frac{1}{1+\alpha^{\prime}}u\in T$
and $\frac{\alpha}{1+\alpha^{\prime}}z\in\mathbb{R}_{+}T\subset T^{+}$,
hence $\alpha z-\alpha^{\prime}z^{\prime}\in T^{+}$.

If, in addition, $T$ is maximal monotone then either from $\operatorname*{aff}T$
monotone or from $T=T^{+}=\operatorname*{aff}T$ we get that $T$
is affine. \end{proof}

\begin{remark} For $T$ monotone and convex the sets $\operatorname*{aff}T$
and $T^{+}$ can be very different. For example, take $T:=\{(0,0)\}\subset\mathbb{R}\times\mathbb{R}$.
Then $\operatorname*{aff}T=\{(0,0)\}$ while $T^{+}$ covers quadrants
I and III.

Also, we cannot expect $T^{+}$ to be convex or a cone. Take $T:=\{(x,0)\mid0\le x\le1\}\subset\mathbb{R}\times\mathbb{R}$.
Then $T^{+}=\{(x,y)\mid x\le0,y\le0\}\cup T\cup\{(x,y)\mid x\ge1,y\ge0\}$
is neither a cone nor convex since for example $(1,1),(0,0)\in T^{+}$
and $(1/2,1/2)\not\in T^{+}$.

However it is easily checked that $T^{+}$ is a (double-)cone whenever
$T$ is a (double-)cone. \end{remark}

It is clear that every maximal monotone operator is dual-representable
with its Fitzpatrick or Penot functions as (d-)representatives. The
converse of the previous fact is currently an open problem. In reflexive
spaces dual-representable operators are maximal monotone (see e.g.\ \cite[Th.\ 3.1]{Burachik/Svaiter:03}).
Under the uniqueness property dual-representable operators become
maximal even in locally convex spaces.

\begin{theorem} \label{dr-unic} The operator $T:X\rightrightarrows X^{*}$
is maximal monotone iff $T$ is dual-representable and unique. \end{theorem}

\begin{proof} While the direct implication is clear (see e.g.\ Theorem
\ref{caracter-max} (ii)), for the converse let $T$ be unique and
dual-representable with a d-representative $h$. From (\ref{r7})
we have that $\varphi_{T}\le h$. For the maximality of $T$ it suffices
to prove that $T$ is NI or $T^{+}=[\varphi_{T}\le c]\subset[\varphi_{T}=c]$.

Assume by contradiction that there exists $z_{0}\in T^{+}\setminus[\varphi_{T}=c]=[\varphi_{T}<c]$.
Then \begin{equation}
\operatorname*{dom}h\subset\{z_{0}\}^{+}=\{z\in Z\mid c(z-z_{0})\ge0\}.\label{m4}\end{equation}
 Indeed, if there is a $z\in\operatorname*{dom}h$ $(\subset\operatorname*{dom}\varphi_{T})$
such that $c(z-z_{0})<0$ then $z\not\in T^{+}$ since $T^{+}$ is
monotone and $z_{0}\in T^{+}$. This implies that $z\in[\varphi_{T}>c]\cap\operatorname*{dom}\varphi_{T}$.
From the continuity of $\varphi_{T}-c$ on the segment $[z,z_{0}]:=\{tz+(1-t)z_{0}\mid0\le t\le1\}$
there exists $t\in(0,1)$ such that $w:=tz+(1-t)z_{0}\in[\varphi_{T}=c]\subset T^{+}$.
It follows that $c(w-z_{0})=t^{2}c(z-z_{0})<0$. This contradicts
the monotonicity of $T^{+}$. Therefore (\ref{m4}) holds.

The inclusion in (\ref{m4}) yields that\begin{equation}
h(z)\ge c(z)-c(z-z_{0})=z\cdot z_{0}-c(z_{0})\quad\forall z\in Z,\label{m5}\end{equation}
 or $c(z_{0})\ge z\cdot z_{0}-h(z)$, for every $z\in Z$. After passing
to supremum over $z\in Z$ we find that $z_{0}\in[h^{\square}=c]=T\subset[\varphi_{T}=c]$.
This is in contradiction with the choice of $z_{0}$. Hence $T$ is
NI and consequently maximal monotone. \end{proof}

\begin{corollary} \label{unic dr ext} The monotone operator $T:X\rightrightarrows X^{*}$
admits a unique dual-representable extension iff $T$ admits a unique
maximal monotone extension. \end{corollary}

\begin{proof} The direct implication is plain since every maximal
monotone operator is dual-representable. For the converse implication
let $h\in\mathcal{D}$ be such that $T\subset[h=c]=:D$. Since $T$
is unique, so is $D$, hence, according to Theorem \ref{dr-unic},
$D$ is maximal monotone. Therefore $D$ is the unique maximal monotone
and dual-representable extension of $T$. \end{proof}

\begin{remark} \label{non}Note that a unique and non-maximal monotone
operator is not dual-representable while a representable and non-maximal
monotone operator is not of NI type. Therefore a unique representable
non-maximal monotone operator is neither NI nor dual-representable.
\end{remark}

In view of Theorem \ref{dr-unic} a path to a proof of the Rockafellar
Conjecture is provided by the following result.

\begin{corollary} \label{Rockafellar}Let $X,Y$ be Banach spaces,
$A:X\rightarrow Y$ a continuous linear operator, $M\in\mathfrak{M}(X)$
and $N\in\mathfrak{M}(Y)$ with $0\in{}^{ic}(\operatorname*{conv}(\operatorname*{dom}N-A(\operatorname*{dom}M)))$.
Then $M+A^{\top}NA$ is maximal monotone iff $M+A^{\top}NA$ is unique.
Here $A^{\top}:Y^{*}\rightarrow X^{*}$ stands for the adjoint of
$A$. \end{corollary}

\begin{proof} As seen in \cite[Remark 1]{VZ mmcs 2008} the mentioned
qualification constraint ensures the dual representability of $M+A^{\top}NA$,
a d-representative for it being given by \[
h^{\square}(x,x^{\ast})=\min\big\{f^{\square}(x,x^{\ast}-A^{\top}y^{\ast})+g^{\square}(Ax,y^{\ast})\mid y^{\ast}\in Y^{\ast}\big\},\quad\forall(x,x^{\ast})\in X\times X^{\ast},\]
 where $f\in\mathcal{R}_{M}$, $g\in\mathcal{R}_{N}$. The stated
equivalence follows directly from Theorem \ref{dr-unic}. \end{proof}

%\strut

\begin{proposition} \label{unic not-NI}Let $T$ be a monotone operator.
Then $T$ is unique and not of NI type iff $\operatorname*{dom}\varphi_{T}$
is monotone and $[\varphi_{T}=c]\subsetneq\operatorname*{dom}\varphi_{T}$.
In this case $[\psi_{T}=c]$ is unique representable and neither NI
nor dual representable and $\operatorname*{dom}\varphi_{T}$ is affine
and the unique maximal monotone extension of both $T$ and $[\psi_{T}=c]$.
\end{proposition}

\begin{proof} For the direct implication, since $T$ is not of NI
type the set $[\varphi_{T}<c]$ is non-empty; hence $[\varphi_{T}=c]\neq\operatorname*{dom}\varphi_{T}$.
Let $z_{0}\in[\varphi_{T}<c]$, $z\in\operatorname*{dom}\varphi_{T}$
and $z_{t}=(1-t)z_{0}+tz$ for $0\le t\le1$. The function $[0,1]\ni t\rightarrow(\varphi_{T}-c)(z_{t})$
is continuous and $\lim_{t\downarrow0}(\varphi_{T}-c)(z_{t})=(\varphi_{T}-c)(z_{0})<0$.
This shows that for every $z\in\operatorname*{dom}\varphi_{T}$ there
is $\delta\in(0,1]$ such that $(1-t)z_{0}+tz\in[\varphi_{T}\le c]$
for every $t\in[0,\delta]$.

Therefore, for $z_{1},z_{2}\in\operatorname*{dom}\varphi_{T}$ there
is $\lambda\in(0,1)$ such that $(1-\lambda)z_{0}+\lambda z_{1},(1-\lambda)z_{0}+\lambda z_{2}\in[\varphi_{T}\le c]$.
This yields that $c(z_{1}-z_{2})\ge0$, since $T$ is unique or equivalently
$T^{+}=[\varphi_{T}\le c]$ is (maximal) monotone. We proved that
$\operatorname*{dom}\varphi_{T}$ is monotone. But $T^{+}\subset\operatorname*{dom}\varphi_{T}$
and $T^{+}$ is maximal monotone. Hence $T^{+}=\operatorname*{dom}\varphi_{T}$
is maximal monotone with a convex graph thus affine (see e.g.\ \cite[Theorem 4.2]{Bauschke Wang Yao 2008}
or Proposition \ref{aff-conv}).

Conversely, whenever $\operatorname*{dom}\varphi_{T}$ is monotone
we get, as above, that $T^{+}=\operatorname*{dom}\varphi_{T}$ is
the unique maximal monotone extension of $T$ and $\varphi_{T}(z)\le c(z)$
for every $z\in\operatorname*{dom}\varphi_{T}$. Together with $[\varphi_{T}=c]\subsetneq\operatorname*{dom}\varphi_{T}$
this shows that $[\varphi_{T}<c]\neq\emptyset$, i.e., $T$ is not
NI.

The remaining conclusions follow from Remark \ref{non}. \end{proof}

\begin{corollary} \label{dom fiT}Let $T\subset Z$ be monotone.
Then $\operatorname*{dom}\varphi_{T}$ is monotone iff $\operatorname*{dom}\varphi_{T}$
is maximal monotone. In this case $T$ is unique and $\operatorname*{dom}\varphi_{T}$
is the unique maximal monotone extension of $T$. \end{corollary}

\begin{remark} Since every unique non-NI operator has its unique
maximal monotone extension an affine operator it is natural to search
for such an object in the class of linear monotone operators (see
Example \ref{unicrep notNIdr} below). \end{remark}

A monotone operator $T$ is not of NI type iff $T$ is not unique
or $T$ is unique but not NI. This allows us to characterize the NI
operator class.

\begin{proposition} \label{NI} Let $T\subset Z$ be monotone. Then
$T$ is NI iff either $[\varphi_{T}\le c]$ is monotone and $\operatorname*{dom}\varphi_{T}$
is non-monotone or $[\varphi_{T}=c]=\operatorname*{dom}\varphi_{T}$.
\end{proposition}

\begin{proof} If $T$ is NI then $T$ is unique, i.e., $T^{+}=[\varphi_{T}\le c]$
is monotone. If $\operatorname*{dom}\varphi_{T}$ is non-monotone
then we are done while if $\operatorname*{dom}\varphi_{T}$ is monotone
then as previously seen $[\varphi_{T}\le c]=[\varphi_{T}=c]=\operatorname*{dom}\varphi_{T}$.

For the converse implication, notice that in both cases $T$ is unique.
The conclusion follows immediately using Proposition \ref{unic not-NI}.
\end{proof}

\begin{remark}\label{rem5} For a monotone operator $T\subset Z$,
the condition $[\varphi_{T}=c]=\operatorname*{dom}\varphi_{T}$ is
equivalent to $\varphi_{T}=c_{M}$, for some affine maximal monotone
set $M\subset Z$. In this case $M=[\varphi_{T}=c]=\operatorname*{dom}\varphi_{T}$
is the unique maximal monotone extension of $T$. \end{remark}

In terms of its Fitzpatrick and Penot functions a monotone operator
$T$ is maximal monotone iff $T=[\psi_{T}=c]$ and either $[\varphi_{T}\le c]$
is monotone and $\operatorname*{dom}\varphi_{T}$ is non-monotone
or $[\varphi_{T}=c]=\operatorname*{dom}\varphi_{T}$.

\strut

Sections \ref{sec:Skew}, \ref{sec:Linear} below deal with strengthening
these results in the (skew-) linear cases.

\section{\label{sec:Skew}Skew double-cones and skew linear subspaces}

As in the previous section, in this section $X$ denotes a separated
locally convex space if not otherwise explicitly specified. A subset
$A$ of $Z=X\times X^{\ast}$ is called \emph{non-negative} if $c(z)\geq0$,
for every $z\in A$ and \emph{skew} if $c(z)=0$, for every $z\in A$.
Our main goal in this section is to study skew linear subspaces and
skew double-cones of $Z$ via their Fitzpatrick and Penot functions.

An important example of monotone operator (multi\-function) is provided
by non-negative linear subspaces $L\subset Z=X\times X^{\ast}$, that
is, $L$ is a linear space with $c(z)\geq0$ for every $z\in L$ (one
obtains immediately that $L$ is monotone) and its particular case
formed by skew linear subspaces of $Z$.

For $A\subset Z$ we set \begin{align*}
-A & :=\{(x,x^{\ast})\in Z\mid(x,-x^{\ast})\in A\},\\
A^{\perp} & :=\{z^{\prime}\in Z\mid\langle z,z^{\prime}\rangle=0,\ \forall z\in A\},\end{align*}
 and \begin{equation}
A_{0}:=\{z\in A^{\perp}\mid c(z)=0\}=A^{\perp}\cap\lbrack c=0].\label{def A0}\end{equation}
 It is obvious that $A^{\perp}=(\operatorname*{cl}_{w\times w^{\ast}}(\operatorname*{lin}A))^{\perp}$,
$A^{\perp\perp}=\operatorname*{cl}_{w\times w^{\ast}}(\operatorname*{lin}A)$,
$A^{\perp}$ is a $w\times w^{\ast}-$closed linear subspace of $Z$,
$A_{0}$ is a double-cone, and $\operatorname*{lin}A=\operatorname*{aff}A=\operatorname*{conv}A$
whenever $A$ is a double-cone. Recall that $w$ denotes the weak
topology on $X$ and $w^{\ast}$ denotes the weak-star topology on
$X^{\ast}$. Also, note that $A\subset A^{\perp}$ implies that $A$
is skew.

\begin{proposition} \label{skewdc FP} Let $D\subset Z$ be a skew
double-cone. Then

\emph{(i)} $\varphi_{D}=\iota_{D^{\perp}}$ and $\psi_{D}=\iota_{\operatorname*{cl}_{w\times w^{\ast}}(\operatorname*{lin}D)}$;

\emph{(ii)} $z$ is m.r.t. $D$ iff $z\in D^{\perp}$ and $c(z)\ge0$,
or equivalently\begin{equation}
D^{+}=[\varphi_{D}\leq c]=\{z\in D^{\perp}\mid c(z)\geq0\};\label{41}\end{equation}

\emph{(iii)} $[\psi_{D}=c]=\{z\in\operatorname*{cl}_{w\times w^{\ast}}(\operatorname*{lin}D))\mid c(z)=0\}=(D^{\perp})_{0}$;
\end{proposition}

\begin{proof} (i) Since $D$ is a skew double-cone we have \[
\varphi_{D}(z)=\sup\{z\cdot z^{\prime}-c(z^{\prime})\mid z^{\prime}\in D\}=\sup\{z\cdot z^{\prime}\mid z^{\prime}\in D\}=\iota_{D^{\perp}}.\]

Assertions (ii) and (iii) follow from the expressions of $\varphi_{D}$
and $\psi_{D}$ in (i) and (\ref{def A0}). \end{proof}

\strut

In the next result we provide a characterization of skew monotone
double-cones.

\begin{proposition} \label{skewdc-char} Let $D\subset Z$ be a double-cone.
TFAE: \emph{(a)} $D$ is skew and monotone, \emph{(b)}$~c(z+z^{\prime})=0$
for all $z,z^{\prime}\in D$, \emph{(c)} $D\subset D^{\perp}$, \emph{(d)}
$\operatorname*{lin}D$ is a skew linear space, \emph{(e)} $\operatorname*{cl}_{w\times w^{\ast}}(\operatorname*{lin}D)$
is a skew linear space.

In particular a linear subspace $S\subset Z$ is skew iff $S\subset S^{\perp}$;
moreover $\operatorname*{cl}_{w\times w^{\ast}}S$ is skew, for every
skew linear subspace $S\subset Z$. \end{proposition}

\begin{proof} Set $L:=\operatorname*{cl}_{w\times w^{\ast}}(\operatorname*{lin}D)$;
of course, $L$ is a linear subspace. The implications (e)~$\Rightarrow$
(d)~$\Rightarrow$ (b)~$\Rightarrow$ (a) are obvious.

(a)~$\Rightarrow$ (b) Because $D$ is skew and monotone, by Proposition
\ref{skewdc FP} we have $\iota_{L}=\psi_{D}\geq c$; hence $c(z)\leq0$
for every $z\in L$. Take $z,z^{\prime}\in D\subset L$; it follows
that $z+z^{\prime}\in L$, and so $c(z+z^{\prime})=c(z)+z\cdot z^{\prime}+c(z^{\prime})=z\cdot z^{\prime}\leq0$.
Since $-z\in D$ we get $z\cdot z^{\prime}=0$, and so $c(z+z^{\prime})=0$.

(b)~$\Rightarrow$ (c) In the proof of (a)~$\Rightarrow$ (b) we
obtained that $z\cdot z^{\prime}=0$ for all $z,z^{\prime}\in D$,
and so $D\subset D^{\perp}$.

(c)~$\Rightarrow$ (b) We have that $z\cdot z^{\prime}=0$ for all
$z,z^{\prime}\in D$. Taking $z^{\prime}=z$ we get $c(z)=0$ for
$z\in D$. Then $c(z+z^{\prime})=c(z)+z\cdot z^{\prime}+c(z^{\prime})=z\cdot z^{\prime}=0$
for all $z,z^{\prime}\in D$.

(b)~$\Rightarrow$ (e). Since $D$ is a double-cone, from (b)~$\Rightarrow$
(a) we have that $D$ and $-D$ are skew and monotone. As in the proof
of (a)~$\Rightarrow$ (b) we obtain that $c(z)\leq0$ for all $z\in L$
and $c(z^{\prime})\leq0$ for all $z^{\prime}\in-L$; therefore, $c(z)=0$
for every $z\in L$, that is, L is skew. The proof is complete. \end{proof}

\begin{corollary} \label{skewmondc}Let $D\subset Z$ be a skew monotone
double-cone. Then $L:=\operatorname*{cl}_{w\times w^{\ast}}(\operatorname*{lin}D)$
is a skew $w\times w^{\ast}$--closed linear space, $\psi_{D}=\psi_{L}=\iota_{L}$,
$\varphi_{D}=\varphi_{L}=\iota_{L^{\perp}}$ and $[\psi_{D}=c]=\operatorname*{dom}\psi_{D}=L$.
In particular, $D$ is representable iff $D=L$, that is, $D$ is
linear and $w\times w^{\ast}$--closed. \end{corollary}

\begin{proof} Because $D^{\perp}=L^{\perp}$, using Proposition \ref{skewdc FP}
(i) we obtain that $\psi_{D}=\psi_{L}=\iota_{L}$ and $\varphi_{D}=\varphi_{L}=\iota_{L^{\perp}}$.
Moreover, by Proposition \ref{skewdc-char} we have that $L$ is a
skew linear space. It follows that $L\subset\lbrack\psi_{L}=c]=[\psi_{D}=c]\subset\operatorname*{dom}\psi_{D}=\operatorname*{dom}\psi_{L}=L$.
Hence $L=[\psi_{L}=c]$, and so $L$ is representable. If $D$ is
representable then $D=[\psi_{D}=c]=L$. The conclusion follows. \end{proof}

\strut

An immediate consequence of the previous result is that the affine
hull of a skew monotone double-cone remains skew (and implicitly monotone).
In section \ref{sec:Linear} we will see that there exists a monotone
double-cone whose affine hull is not monotone. Moreover, Proposition
\ref{skewdc-char} and Corollary \ref{skewmondc} show that when dealing
with skew monotone double-cones, we may assume without loss of generality
that they are linear spaces, even $w\times w^{\ast}$--closed linear
spaces. For this reason in the sequel, in this section, we work only
with skew linear spaces.

\begin{proposition} \label{skew NI}Let $S\subset Z$ be a skew linear
space. TFAE: \emph{(a)} $S$ is NI, \emph{(b)}~$\operatorname*{cl}_{w\times w^{\ast}}S$
is NI, \emph{(c)} $\operatorname*{cl}_{w\times w^{\ast}}S$ is maximal
monotone, \emph{(d)} $S$ has a unique representable extension, \emph{(e)}
$-S^{\perp}$ is monotone.

In this case \begin{equation}
\operatorname*{cl}{}_{w\times w^{\ast}}S=[\psi_{S}=c]=[\varphi_{S}\leq c]=[\varphi_{S}=c]=\{z\in S^{\perp}\mid c(z)=0\}=S_{0},\label{42}\end{equation}
 $S$, $-S$, $-\operatorname*{cl}_{w\times w^{\ast}}S$ are unique
in $Z$, $\operatorname*{cl}_{w\times w^{\ast}}S$ is the unique maximal
monotone extension and the unique representable extension of $S$
in $Z$, and $-S^{\perp}$ is the unique maximal monotone extension,
as well as the unique dual-representable extension of $-S$ and of
$-\operatorname*{cl}_{w\times w^{\ast}}S$ in $Z$. \end{proposition}

\begin{proof} By Corollary \ref{skewmondc} we have that $L:=\operatorname*{cl}{}_{w\times w^{\ast}}S=[\psi_{S}=c]$.
The equivalence of conditions (a), (b) and (d) follows from Proposition
\ref{caracter-NI} (ii), while the equivalence of (b) and (e) follows
immediately from the relation $\varphi_{S}=\varphi_{L}=\iota_{L^{\perp}}=\iota_{S^{\perp}}$
mentioned in Corollary \ref{skewmondc}. In this case, from Proposition
\ref{caracter-NI} (iii) and Corollary \ref{skewmondc}, we find (\ref{42}).

Assume now that $S$ is NI. By Proposition \ref{caracter-NI} (iii)
$S$ is unique with $L$ its unique maximal monotone extension and
unique representable extension.

Since $S$ is NI we have that $-S^{\perp}=-L^{\perp}$ is monotone.
Because $S$ is skew, so is $-S$, and so $(-S)^{+}=[\varphi_{-S}\leq c]=[\varphi_{-L}\leq c]=[\iota_{-L^{\perp}}\leq c]=-L^{\perp}$
is monotone. According to Proposition \ref{caracter-unic} and Corollary
\ref{unic dr ext}, $-L^{\perp}$ is the unique maximal monotone extension
and the unique dual representable extension of $-S$ and $-L$ in
$Z$. \end{proof}

\strut

Direct consequences of the previous result follow.

\begin{corollary} \label{skew skew} Let $S\subset Z$ be a skew
linear space such that $S^{\perp}$ is skew. Then $S$ is NI, $\operatorname*{cl}{}_{w\times w^{\ast}}S=S^{\perp}$
and $S^{\perp}$ is the unique maximal monotone extension of $S$
in $Z$. \end{corollary}

\begin{proposition} \label{skew unic} Let $S\subset Z$ be a skew
linear space and $S_{0}:=[\varphi_{S}=c]$. TFAE: \emph{(i)} $S$
is unique, \emph{(ii)} $S_{0}$ is monotone, \emph{(iii)} $S^{\perp}$
is monotone or $-S^{\perp}$ is monotone, \emph{(iv)} $S_{0}$ is
linear, \emph{(v)}$~S_{0}$ is convex, \emph{(vi)} $\operatorname*{cl}{}_{w\times w^{\ast}}S$
is unique.

In this case $S_{0}=\operatorname*{cl}{}_{w\times w^{\ast}}S$. \end{proposition}

\begin{proof} By Corollary \ref{skewmondc} we have that $L:=\operatorname*{cl}{}_{w\times w^{\ast}}S$
is skew and $\varphi_{S}=\varphi_{L}$; hence (i) $\Leftrightarrow$
(vi). Moreover, the implication (iv) $\Rightarrow$ (v) is obvious.

(i) $\Rightarrow$ (ii) According to Proposition \ref{caracter-unic}
it is clear that $S$ unique implies that $S_{0}$ $(\subset\lbrack\varphi_{S}\le c])$
is monotone.

(ii) $\Rightarrow$ (iii) For this assume that $S^{\perp}$ and $-S^{\perp}$
are not monotone. Then there exist $z,z^{\prime}\in S^{\perp}$ such
that $c(z^{\prime})<0<c(z)$. For $t\in\mathbb{R}$ and $\eta(t):=c(z+tz^{\prime})=c(z)+t\langle z,z^{\prime}\rangle+t^{2}c(z^{\prime})$
we have that $\eta(0)>0$ and $\lim_{|t|\rightarrow\infty}\eta(t)=-\infty$.
Therefore there exist distinct $t_{1},t_{2}\in\mathbb{R}$ such that
$\eta(t_{1})=\eta(t_{2})=0$, and so $z_{1}:=z+t_{1}z^{\prime},\ z_{2}:=z+t_{2}z^{\prime}\in S_{0}$.
Since $c(z_{1}-z_{2})=(t_{1}-t_{2})^{2}c(z^{\prime})<0$ we see that
$S_{0}$ is not monotone.

The implication (iii) $\Rightarrow$ (i) is straightforward from the
second part of Proposition \ref{skew NI} applied for $S$ or $-S$.

(iii) $\Rightarrow$ (iv) If $-S^{\perp}$ is monotone, using by Proposition
\ref{skew NI} we obtain that $S$ is NI and $S_{0}=L$ (see (\ref{42})).
Similarly, if $S^{\perp}$ is monotone then $-S_{0}=(-S)_{0}=-L$.
Hence in both cases $S_{0}=L$, and so $S_{0}$ is linear.

(v) $\Rightarrow$ (ii) If $S_{0}$ is convex then $\iota_{S_{0}}$
is convex, $\iota_{S_{0}}\geq c$, and $S_{0}=[\iota_{S_{0}}=c]$;
hence $S_{0}$ is monotone. The proof is complete. \end{proof}

\begin{proposition} \label{skew non-unic}Let $S\subset Z$ be a
skew linear space which is not unique. Then $S_{0}=[\varphi_{S}=c]$
is an NI double cone which is neither monotone nor convex, $S^{\perp}$
and $-S^{\perp}$ are not monotone, $\operatorname*{cl}{}_{w\times w^{\ast}}S$
is not unique, $(S_{0})^{\perp}$ is skew, $\operatorname*{cl}{}_{w\times w^{\ast}}\operatorname*{conv}S_{0}=S^{\perp}$,
and \begin{equation}
S^{++}=(\operatorname*{cl}{}_{w\times w^{\ast}}S)^{++}=(S_{0})^{\perp}=\operatorname*{cl}{}_{w\times w^{\ast}}S\subsetneq S_{0}\subsetneq S^{+}.\label{43}\end{equation}

\end{proposition}

\begin{proof} Because $\varphi_{S}=\iota_{S^{\perp}}$ and $S_{0}=[\varphi_{S}=c]$,
it is clear that $S_{0}$ is a skew double-cone. Because $S\subset Z$
is a skew linear space which is not unique, by the preceding proposition,
we know that $L:=\operatorname*{cl}{}_{w\times w^{\ast}}S$ is not
unique, $S_{0}$ is not monotone or convex, and $S^{\perp}$, $-S^{\perp}$
are not monotone. Because $S\subset\lbrack\psi_{S}=c]\subset\lbrack\varphi_{S}=c]=S_{0}\subset S^{\perp}$
we get $L=S^{\perp\perp}\subset(S_{0})^{\perp}\subset S^{\perp}$.
Since $S_{0}$ is a skew double-cone we have $\varphi_{S_{0}}=\iota_{(S_{0})^{\perp}}$.

Assume, by contradiction, that $S_{0}$ is not NI. Then, there is
$z\in\lbrack\varphi_{S_{0}}<c]$, that is, $z\in(S_{0})^{\perp}\subset S^{\perp}$
with $c(z)>0$. Clearly $z$ is m.r.t.\ $S_{0}$. Because $S^{\perp}$
is not monotone there is $z^{\prime}\in S^{\perp}$ such that $c(z^{\prime})<0$.
Let $\eta(t):=c(z+tz^{\prime})$ for $t\in\mathbb{R}$. Note that
$\eta(0)=c(z)>0$ and $\lim_{\left\vert t\right\vert \rightarrow\infty}\eta(t)=-\infty$.
Therefore there is $t_{0}\neq0$ such that $\eta(t_{0})=0$, i.e.,
$z_{0}:=z+t_{0}z^{\prime}\in S_{0}$. Since $z$ is m.r.t.\ $S_{0}$,
we get the contradiction $c(z_{0}-z)=t_{0}^{2}c(z^{\prime})\geq0$.
Therefore $S_{0}$ is NI and this translates into $\varphi_{S_{0}}\geq c$
or $-(S_{0})^{\perp}$ is monotone.

From the equivalence of (i) and (iii) in Proposition \ref{skew unic}
we have that $-S$ is not unique, and so, by the above argument, $-((-S)_{0})^{\perp}=(S_{0})^{\perp}$
is monotone. Hence $(S_{0})^{\perp}$ is skew.

From Corollary \ref{skewmondc} and (\ref{r3}) we have \[
S\subset L=[\psi_{S}=c]\subset\lbrack\varphi_{S}=c]=S_{0}\subset S^{+}.\]
 This leads to (see Lemma \ref{ineq A+} (i))\begin{equation}
\iota_{L}=\psi_{S}\geq\varphi_{S^{+}}\geq\varphi_{S_{0}}=\iota_{(S_{0})^{\perp}},\label{44}\end{equation}
 whence $L\subset(S_{0})^{\perp}\subset S^{\perp}$.

Assume that $z\in(S_{0})^{\perp}\setminus L$. By a separation theorem,
there is $w\in S^{\perp}$ such that $z\cdot w\neq0$. Since $z\in(S_{0})^{\perp}$
$(\subset S^{\perp})$ we know that $w\not\in S_{0}$, whence $c(w)\neq0$.
Notice that, since $(S_{0})^{\perp}$ is skew, $c(z)=0$ and so\[
c(z+\lambda_{0}w)=\lambda_{0}z\cdot w+\lambda_{0}^{2}c(w)=0,\ \mathrm{for}\ \lambda_{0}:=-\frac{1}{c(w)}z\cdot w\neq0,\]
 that is $z+\lambda_{0}w\in S_{0}$. From $z\in(S_{0})^{\perp}$ we
get the contradiction $0=z\cdot(z+\lambda_{0}w)=\lambda_{0}z\cdot w\neq0$.
Therefore $L=(S_{0})^{\perp}$, and by (\ref{44}), this yields $\psi_{S}=\varphi_{S^{+}}=\iota_{L}$.
Since $L$ is skew and $S^{+}$ is NI by Lemma \ref{ineq A+} (ii),
we deduce $S^{++}=[\varphi_{S^{+}}\leq c]=[\varphi_{S^{+}}=c]=L$.
By Proposition \ref{caracter-NI} (i) we have $S^{+}=[\psi_{S}=c]^{+}=L^{+}$,
and so $S^{++}=L^{++}$.

The strict inclusion $L\subsetneq S_{0}$ comes from the fact that
$L$ is a skew subspace while $S_{0}$ is not monotone, while the
strict inclusion $S_{0}\subsetneq S^{+}$ comes from the fact that
$-S^{\perp}$ is not monotone. The other conclusions follow from Proposition
\ref{skew unic} while from $(S_{0})^{\perp}=\operatorname*{cl}{}_{w\times w^{\ast}}S$
we find $\operatorname*{cl}{}_{w\times w^{\ast}}\operatorname*{conv}S_{0}=(S_{0})^{\perp\perp}=S^{\perp}$.
\end{proof}

\strut

We comprise some of the information on a skew linear space (or skew
monotone double-cone) $S\subset Z$ in the following chart:

\begin{center}
\begin{tabular}{l}
1. $S$ is representable iff $S$ is linear and $w\times w^{\ast}-$closed\tabularnewline
2. $S$ is NI iff $-S^{\perp}$ is monotone\tabularnewline
3. $S$ is unique iff $S^{\perp}$ is monotone or $-S^{\perp}$ is
monotone\tabularnewline
4. $S$ is unique and not NI iff $S^{\perp}$ is monotone and non-skew\tabularnewline
5. $S$ is maximal monotone iff $S$ is linear and $w\times w^{\ast}-$closed,
and $-S^{\perp}$ is monotone\tabularnewline
\end{tabular}
\par\end{center}

These results on skew linear subspaces prefigure the results we will
obtain in section \ref{sec:Linear} for linear multifunctions.

In the preceding results the topology $w\times w^{*}$ on $Z$ can
be replaced by any locally convex topology $\sigma$ compatible with
the natural duality $(Z,Z)$, that is, for which $(Z,\sigma)^{*}=Z$,
for example the topology $\tau\times w^{*}$, where $\tau$ is the
initial topology of $X$.

\begin{example} \label{unicrep notNIdr} In \cite[page 89]{Gossez72}
Gossez considered the linear operator $T:\ell_{1}\rightarrow\ell_{\infty}$
defined by $T\left((x_{n})_{n\geq1}\right):=(y_{n})_{n\geq1}$ with
$y_{n}:=\sum_{k\geq1}x_{k}+x_{n}-2\sum_{k=1}^{n}x_{k}$. The operator
$T$ is skew, that is, $\left\langle x,Tx\right\rangle =0$ for every
$x\in\ell_{1}$, or equivalently, $\left\langle x,Ty\right\rangle =-\left\langle y,Tx\right\rangle $
for all $x,y\in\ell_{1}$ (in the duality of $\ell_{1}\times\ell_{\infty}$).
In fact for $x\in\ell_{1}$ we have that $Tx\in\mathfrak{c}$, the
subspace of convergent sequences of $\ell_{\infty}$. Indeed, $\lim Tx=-\sum_{k\geq1}x_{k}=-\left\langle x,e\right\rangle $,
where $e_{n}:=1$ for $n\geq1$. Consequently, $T_{1}x:=Tx+\left\langle x,e\right\rangle e\in\mathfrak{c}_{0}$
for every $x\in\ell_{1}$. From Proposition \ref{p-lin-pos} it follows
that both $\{(T_{1}x,x)\mid x\in\ell_{1}\}\subset\mathfrak{c}_{0}\times\ell_{1}$
and $\{(x,T_{1}x)\mid x\in\ell_{1}\}\subset\ell_{1}\times\ell_{\infty}$
are maximal monotone in the corresponding spaces.

Consider \begin{align}
S & :=\left\{ (Tx,x)\mid x\in\ell_{1},\ \left\langle x,e\right\rangle =0\right\} \subset\mathfrak{c}_{0}\times\ell_{1},\label{sk1}\\
R & :=-S=\left\{ (-Tx,x)\mid x\in\ell_{1},\ \left\langle x,e\right\rangle =0\right\} .\label{sk2}\end{align}
 Let us first determine $S^{\perp}$ (in $\mathfrak{c}_{0}\times\ell_{1}$).
Consider $(u,v)\in S^{\perp}$. Then $\left\langle u,x\right\rangle _{\mathfrak{c}_{0}\times\ell_{1}}+\left\langle Tx,v\right\rangle _{\mathfrak{c}_{0}\times\ell_{1}}=0$
for every $x\in\ell_{1}$ with $\left\langle x,e\right\rangle =0$.
Since $\left\langle u,x\right\rangle _{\mathfrak{c}_{0}\times\ell_{1}}=\left\langle x,u\right\rangle $
and $\left\langle Tx,v\right\rangle _{\mathfrak{c}_{0}\times\ell_{1}}=\left\langle v,Tx\right\rangle =-\left\langle x,Tv\right\rangle $,
we obtain that \[
\left[x\in\ell_{1},\ \left\langle x,e\right\rangle =0\right]\Rightarrow\left\langle x,u-Tv\right\rangle =0,\]
 and so necessarily $u-Tv=\gamma e$, that is, $u=Tv+\gamma e$, for
some $\gamma\in\mathbb{R}$. Since $u\in\mathfrak{c}_{0}$, it follows
that $\gamma=\left\langle v,e\right\rangle $. Hence $S^{\perp}\subset\left\{ (Tv+\left\langle v,e\right\rangle e,v)\mid v\in\ell_{1}\right\} $.
Conversely, if $v,x\in\ell_{1}$ and $\left\langle x,e\right\rangle =0$
then \[
\left\langle Tx,v\right\rangle _{\mathfrak{c}_{0}\times\ell_{1}}+\left\langle Tv+\left\langle v,e\right\rangle e,x\right\rangle _{\mathfrak{c}_{0}\times\ell_{1}}=\left\langle v,Tx\right\rangle +\left\langle x,Tv\right\rangle +\left\langle v,e\right\rangle \left\langle x,e\right\rangle =0,\]
 which proves that \[
S^{\perp}=\left\{ (Tv+\left\langle v,e\right\rangle e,v)\mid v\in\ell_{1}\right\} .\]

Moreover, $S=S^{\perp\perp}$. Indeed, because $S\subset S^{\perp\perp},$
let us prove the converse inclusion. Consider $(y,x)\in S^{\perp\perp}$.
Then \[
0=\left\langle y,v\right\rangle _{\mathfrak{c}_{0}\times\ell_{1}}+\left\langle Tv+\left\langle v,e\right\rangle e,x\right\rangle _{\mathfrak{c}_{0}\times\ell_{1}}=\left\langle v,y\right\rangle +\left\langle x,Tv+\left\langle v,e\right\rangle e\right\rangle =\left\langle y-Tx+\left\langle x,e\right\rangle e,v\right\rangle \]
 for every $v\in\ell_{1},$ and so $y-Tx+\left\langle x,e\right\rangle e=0.$
Taking the limit we get $2\left\langle x,e\right\rangle =0,$ and
so $y=Tx.$ Therefore, $(y,x)\in S,$ and so $S=\operatorname*{cl}_{w\times w^{\ast}}S=S^{\perp\perp}.$
Since $R=-S,$ we obtain that $R=\operatorname*{cl}_{w\times w^{\ast}}R=R^{\perp\perp},$
too.

We know that $\varphi_{S}=\iota_{S^{\perp}}$ and $\varphi_{R}=\iota_{R^{\perp}}=\iota_{-S^{\perp}}.$
Because $\left\langle Tv+\left\langle v,e\right\rangle e,v\right\rangle =\left\langle v,e\right\rangle ^{2}$
and $\left\langle v,e\right\rangle ^{2}>0$ for certain $v$, we get
that $S^{\perp}$ is monotone and non-skew. Therefore, according to
Propositions \ref{skew NI} and \ref{skew unic}, $S$ $(=\operatorname*{cl}_{w\times w^{\ast}}S)$
is representable, unique and not NI with $S^{+}$ $(=S^{\perp}=\operatorname*{dom}\varphi_{S})$
its unique maximal monotone extension; moreover, $S$ is not dual-representable
and $S=[\varphi_{S}=c]=[\iota_{S^{\perp}}=c]$.

On the other hand, because $-R^{\perp}=S^{\perp}$ is monotone, by
Proposition \ref{skew NI} we have that $R$ $(=\operatorname*{cl}_{w\times w^{\ast}}R)$
is maximal monotone. %}

\end{example}

\begin{example} Let us consider the sets $S,R$ defined in (\ref{sk1})
and (\ref{sk2}) as subsets of $\ell_{\infty}\times\ell_{1},$ $\ell_{\infty}$
being endowed with the $w^{\ast}$-topology and $\ell_{1}$ with the
weak topology. Then $S$ and $R$ are skew linear subspaces. The calculus
above shows that $S^{\perp}=\left\{ (Tv+\gamma e,v)\mid v\in\ell_{1},~\gamma\in\mathbb{R}\right\} .$
Then one obtains immediately that $S^{\perp\perp}=S.$ Hence $S=\operatorname*{cl}_{w\times w^{\ast}}S=S^{\perp\perp}$
and $R=\operatorname*{cl}_{w\times w^{\ast}}R=R^{\perp\perp}.$ It
follows that $S$ and $R$ are representable (because $S=\operatorname*{cl}_{w\times w^{\ast}}S$
and $R=\operatorname*{cl}_{w\times w^{\ast}}R$), not unique (because
$S^{\perp}$ and $-S^{\perp}$ are not monotone), and consequently,
not NI. \end{example}

\begin{remark} Because $R$ is maximal monotone in $\mathfrak{c}_{0}\times\ell_{1}$
but it is not NI in $\ell_{\infty}\times\ell_{1},$ it furnishes an
example of maximal monotone operator which is not NI in the sense
of Simons. \end{remark}

%}
%\strut

\section{\label{sec:Linear}Monotone double-cones and monotone linear subspaces}

In the beginning of this section we investigate monotone operators
that admit affine maximal monotone extensions.

\begin{lemma} \label{affine ext}Every affine monotone subset of
$X\times X^{*}$ admits an affine maximal monotone extension. \end{lemma}

\begin{proof} Let $L_{0}$ be affine monotone in $Z$. Without loss
of generality we assume that $0\in L_{0}$. Consider \[
\mathcal{L}=\{L\subset Z\mid L_{0}\subset L,\ L\ \text{is\ a\ linear\ monotone\ subspace}\}\]
 ordered by inclusion. Every chain $\{L_{i}\}_{i}$ in $\mathcal{L}$
admits $\cup_{i}L_{i}$ as an upper bound in $\mathcal{L}$. By the
Zorn Lemma there exists a maximal element in $\mathcal{L}$ denoted
by $L$.

Let $z\in L^{+}$ and $L^{\prime}=L+\mathbb{R}z$. Clearly, $L^{\prime}$
is linear and $L^{\prime}\supset L_{0}$. Moreover, for every $u\in L$
and every $\alpha\in\mathbb{R}$ we have that $c(u+\alpha z)=c(u)\geq0$
if $\alpha=0$ and $c(u+\alpha z)=\alpha^{2}c(z-\alpha^{-1}(-u))\geq0$
if $\alpha\neq0$ since $z\in L^{+}$ and $\alpha^{-1}(-u)\in L$.
Therefore $L^{\prime}\in\mathcal{L}$, and so $L^{\prime}=L$, whence
$z\in L$. This proves that $L$ is maximal monotone. \end{proof}

\begin{theorem} \label{convex ext}Every convex monotone subset of
$X\times X^{*}$ admits an affine maximal monotone extension. \end{theorem}

\begin{proof} Let $T$ be convex monotone in $X\times X^{*}$. According
to Proposition \ref{aff-conv}, $\operatorname*{aff}T$ is (affine
and) monotone thus, by the previous lemma, it admits an affine maximal
monotone extension that is also an extension for $T$. \end{proof}

\strut

The possibility of finding closed forms for the Fitzpatrick and Penot
functions allowed us in the previous section to provide characterizations
of the main classes of skew operators. Similar ideas are used in this
section for the multi-valued linear monotone case; the goal being
to offer criteria for a linear monotone operator to belong to a certain
class in terms of its Penot and Fitzpatrick functions or through the
monotonicity of sets directly associated to the operator via decomposition
and orthogonality.

An important example of linear monotone subspace is that of non-negative
single-valued linear operator, that is, $A:\operatorname*{dom}A=X\rightarrow X^{\ast}$
with $\langle x,Ax\rangle\geq0$ for every $x\in X$; if $\langle x,Ax\rangle=0$
for every $x\in X$ then $A$ is skew. Of course, the linear operator
$A:X\rightarrow X^{\ast}$ is skew iff $A$ and $-A$ are non-negative.
The next two results refer to such situations.

\begin{proposition} \label{p-lin-pos} \emph{(i)} If $A:X\rightarrow X^{\ast}$
is a non-negative linear operator then $A$ is maximal monotone.

\emph{(ii)} If $A^{\prime}:X^{\ast}\rightarrow X$ is a non-negative
linear operator, that is $\langle A^{\prime}x^{\ast},x^{\ast}\rangle\geq0$
for every $x^{\ast}\in X^{\ast}$, then $A^{\prime}$ is maximal monotone.

\emph{(iii)} Assume that $X$ is a Banach space and let $A:X\rightarrow X^{\ast}$
and $A^{\prime}:X^{\ast}\rightarrow X$ be non-negative operators.
Then $A$ and $A^{\prime}$ are continuous (for $X$ and $X^{\ast}$
endowed with the norm topologies). \end{proposition}

\begin{proof} (i) Assume that $A:X\rightarrow X^{\ast}$ is a non-negative
linear operator and let $(u,u^{\ast})\in X\times X^{\ast}$ be m.r.t.\ $A$.
Then $\left\langle u-x,u^{\ast}-Ax\right\rangle \geq0$ for $x\in X$.
Taking $x=u-ty$ with $t>0$, $y\in X$ we get $\left\langle y,u^{\ast}-Au+tAy\right\rangle \geq0$.
Letting $t\rightarrow0$, then replacing $y$ by $-y$ we obtain that
$\left\langle y,u^{\ast}-Au\right\rangle =0$ for every $y\in X$.
Hence $u^{\ast}=Au$, which shows that $A$ is maximal monotone.

(ii) Apply (i) for $A$ replaced by $A^{\prime}$ (or use a similar
argument).

(iii) Assume now that $X$ is a Banach space. Since $A$ is maximal
monotone, then $\mathrm{gph\,}A$ is strongly closed and so $A$ is
continuous. Similarly for $A^{\prime}$. \end{proof}

\strut

Note that the graph of a non-negative linear operator $A:X\rightarrow X^{\ast}$
is a non-negative linear subspace of $Z$ and the graph of a skew
linear operator $A:X\rightarrow X^{\ast}$ is a skew subspace of $Z$.

\begin{proposition} \label{skew-full projection}Let $S\subset Z$
be a skew linear subspace. Then $\Pr_{X}(S)=X$ iff $S$ is the graph
of a skew linear operator $A:X\rightarrow X^{\ast}$. Hence if $\Pr_{X}(S)=X$
then $S=S^{\perp}$ (in particular $S$ is $w\times w^{\ast}$-closed)
and $S$ is maximal monotone. Similarly, $\Pr_{X^{\ast}}(S)=X^{\ast}$
iff $S^{-1}:=\{(x^{\ast},x)\mid(x,x^{\ast})\in S\}$ is the graph
of a skew linear operator $A^{\prime}:X^{\ast}\rightarrow X$. Hence
if $\Pr_{X^{\ast}}(S)=X^{\ast}$ then $S=S^{\perp}$ and $S$ is maximal
monotone. \end{proposition}

\begin{proof} Assume that $\Pr_{X}(S)=X$. Let $u^{*}\in X^{*}$
be such that $(0,u^{*})\in S$. For every $x\in X$ there exists $x^{*}\in X^{*}$
with $(x,x^{*})\in S$. Because $S$ is a linear subspace, $(x,x^{*}+u^{*})\in S$.
Hence $0=\langle x,x^{*}+u^{*}\rangle=\langle x,x^{*}\rangle+\langle x,u^{*}\rangle=\langle x,u^{*}\rangle$.
Since $x\in X$ is arbitrary, we get $u^{*}=0$. It follows that,
for every $x\in X$, $S(x)=\{Ax\}$ is a singleton. The operator $A$
obtained in this way is a skew linear operator. From Proposition \ref{p-lin-pos}
we obtain that $A$ is maximal monotone, and so $S$ is maximal monotone.
Similarly since $\Pr_{X}(-S)=X$, we know that $-S$ is maximal monotone,
therefore $-S$ is (NI). From Proposition \ref{skew NI} we have that
$-(-S)^{\perp}=S^{\perp}$ is (the unique) maximal monotone extension
of $S$, and so $S=S^{\perp}$.

Applying the previous result for $S^{-1}$ in the case $\Pr_{X^{*}}(S)=X^{*}$
or repeating the above argument we get the last assertion. \end{proof}

\strut

To a subset $A\subset Z$ we associate its

\begin{itemize}
\item \emph{skew part} $S_{A}:=\operatorname*{Skew}(A):=\{z\in A\mid c(z)=0\}=A\cap[c=0]$,
\item \emph{positive part} $P_{A}:=\operatorname*{Pos}(A):=\{z\in A\mid c(z)>0\}=A\cap[c>0]$,
\item \emph{negative part} $N_{A}:=\operatorname*{Neg}(A):=\{z\in A\mid c(z)<0\}=A\cap[c<0]$,
\item \emph{unitary part} $U_{A}:=\operatorname*{Unit}(A):=\{z\in A\mid c(z)=1\}=A\cap[c=1]$,
\item \emph{crown} $C_{A}:=\operatorname*{Crown}(A)=\operatorname*{cl}_{w\times w^{\ast}}(\operatorname*{conv}U_{A})$;
\end{itemize}
\noindent clearly $A=S_{A}\cup P_{A}\cup N_{A}$.

We also use the notation $\mathbb{U}=[c=1]$, $\mathbb{P}=[c>0]$,
$\mathbb{S}=[c=0]$, and introduce the map \begin{equation}
\zeta:\mathbb{P}\rightarrow Z,\quad\zeta(z)=z/\sqrt{c(z)}\quad(z\in\mathbb{P}).\label{def zeta}\end{equation}
 Note that $c(\zeta(z))=1$ for every $z\in\mathbb{P}$, whence $\zeta(\mathbb{P})=\mathbb{U}$,
$\zeta(z)=z$ for every $z\in\mathbb{U}$, and $\zeta(tz)=\zeta(z)$
for all $t\in\mathbb{R}^{\ast}$ and $z\in\mathbb{P}$, where $\mathbb{R}^{\ast}:=\mathbb{R}\setminus\{0\}$\emph{.}

Whenever $D$ is a double-cone, $\mathbb{R}S_{D}=S_{D}$, $\mathbb{R}^{\ast}P_{D}=P_{D}=(0,\infty)U_{D}=\mathbb{R}^{\ast}U_{D}$,
$\mathbb{R}^{\ast}N_{D}=N_{D}$. Our analysis is based on the study
of the Fitzpatrick and Penot functions for double-cones. Note that
it is readily seen that whenever $D$ is a double-cone with a nonempty
negative part $N_{D}=D\cap\lbrack c<0]$ its Fitzpatrick function
is identically equal to $+\infty$ while its Penot function is improper;
more precisely $\psi_{D}(z)=-\infty$ for $z\in\operatorname*{cl}_{w\times w^{\ast}}(\operatorname*{conv}D)$
and $\psi_{D}(z)=+\infty$ for $z\in Z\setminus\operatorname*{cl}_{w\times w^{\ast}}(\operatorname*{conv}D)$.
Indeed, taking $z_{0}\in N_{D}$ and $t\in\mathbb{R}$ we have that
\[
\varphi_{D}(z)=\sup\{z\cdot z^{\prime}-c(z^{\prime})\mid z^{\prime}\in D\}\geq\sup\{tz\cdot z_{0}+t^{2}(-c(z_{0}))\mid t\in\mathbb{R}\}=\infty,\]
 for all $z\in Z$. Hence $\varphi_{L}=\infty$ and $\varphi_{L}^{\square}=-\infty$.

That is why, from this point of view, it is natural to study double-cones
$D$ that are non-negative (i.e., with an empty negative part $D\cap[c<0]$
or equivalently $c(z)\ge0$, for every $z\in D$); in particular our
main interests lie in studying monotone double-cones or linear operators.
It is clear that two non-negative double-cones coincide if they have
the same skew and unitary parts.

Recall that the \emph{support function} to $A\subset Z$ is given
by $\sigma_{A}(z)=\sup_{u\in A}z\cdot u$ for $z\in Z$, while $\operatorname*{bar}A:=\operatorname*{dom}\sigma_{A}$
denotes the \emph{barrier cone} of $A$, and $\operatorname*{bar}A=\operatorname*{bar}(\operatorname*{cl}_{w\times w^{\ast}}(\operatorname*{conv}A))$.
By convention $\sup\emptyset:=-\infty$ and $\inf\emptyset:=+\infty$;
hence $\sigma_{\emptyset}=-\infty$. Moreover, the \emph{Minkowski
functional} associated to $A\subset Z$ is given by $p_{A}(z)=\inf\{t>0\mid z\in tA\}$.
It is well-known that, when $A$ is a convex set containing $0$,
$p_{A}$ is a sublinear function whose domain is $\mathbb{R}_{+}A$;
moreover, if $A$ is symmetric then $p_{A}$ is even, and so $p_{A}$
is a semi-norm when restricted to its domain which is a linear space.
When $A$ is a symmetric closed convex set then $tA=[p_{A}\leq t]$
for every $t>0$ and $[p_{A}=0]=A_{\infty}$; in particular $p_{A}$
is lsc. Recall that the \emph{asymptotic cone} of the nonempty closed
convex set $C\subset Z$ is $C_{\infty}=\cap_{t>0}t(C-c)$ for some
(every) $c\in C$ and \emph{the polar} of $C$ is $C^{\circ}:=[\sigma_{C}\leq1]$.
Of course, $C_{\infty}$ is a closed convex cone; $C_{\infty}$ is
a linear space when $C$ is symmetric. Recall also that $D_{0}:=D^{\perp}\cap\lbrack c=0]=\operatorname*{Skew}(D^{\perp})$.
We also use the conventions $0\cdot(-\infty):=0$ and $0\cdot(+\infty):=+\infty$;
therefore $0f=\iota_{\operatorname*{dom}f}$.

\begin{proposition} \label{dc FP} Let $D\subset Z$ be a non-negative
double-cone with skew part $S$, non-empty positive part $P$, unitary
part $U$, and crown $C$. Then

\emph{(i)} $\varphi_{D}=\frac{1}{4}\sigma_{C}^{2}+\iota_{S^{\perp}}$,
or in extended form: \[
\varphi_{D}(z)=\sup\limits _{w\in P}\frac{|z\cdot w|^{2}}{4c(w)}\ \text{if}\ z\in S^{\perp},\quad\varphi_{D}(z)=\infty\ \text{otherwise}.\]

\emph{(ii)} $\operatorname*{dom}\varphi_{D}=S^{\perp}\cap\operatorname*{bar}C$,
and so $\operatorname*{dom}\varphi_{D}$ is a linear subspace of $Z$.

\emph{(iii)} $\varphi_{D}(tz)=t^{2}\varphi_{D}(z)$, for all $z\in Z$
and $t\in\mathbb{R}$.

\emph{(iv)} $\psi_{D}=\operatorname*{cl}_{w\times w^{\ast}}[p_{C}^{2}\square\iota_{S^{\perp}}^{\ast}]$.

\emph{(v)} $\psi_{D}(tz)=t^{2}\psi_{D}(z)$, for all $z\in Z$ and
$t\in\mathbb{R}$, $\operatorname*{dom}\psi_{D}$ is a linear space,
and $\operatorname*{lin}D\subset\operatorname*{dom}\psi_{D}\subset\operatorname*{cl}_{w\times w^{\ast}}(\operatorname*{lin}D)$.

\emph{(vi)} $[\varphi_{D}=0]=D^{\perp},$ $[\sigma_{C}=0]=C^{\perp}=U^{\perp}=P^{\perp},$
$D^{\perp}=S^{\perp}\cap C^{\perp}.$

\emph{(vii)} $D^{+}=D_{0}\cup\left(S^{\perp}\cap\zeta^{-1}(2C^{\circ})\right)$
with $\operatorname*{Skew}(D^{+})=D_{0}:=\operatorname*{Skew}(D^{\perp})$,
$\operatorname*{Unit}(D^{+})=(2C^{\circ})\cap S^{\perp}\cap\mathbb{U}$
and $\operatorname*{Pos}(D^{+})=S^{\perp}\cap\zeta^{-1}(2C^{\circ})$.
\end{proposition}

\begin{proof} (i) Let $z\in Z$. Then \begin{align*}
\varphi_{D}(z) & =\sup\left\{ z\cdot z^{\prime}-c(z^{\prime})\mid z^{\prime}\in S\cup\mathbb{R}U\right\} \\
 & =\max\left(\sup\left\{ z\cdot z^{\prime}\mid z^{\prime}\in S\right\} ,\sup\left\{ \lambda z\cdot z^{\prime}-\lambda^{2}\mid z^{\prime}\in U,\ \lambda\in\mathbb{R}\right\} \right)\\
 & =\max\left(\iota_{S^{\perp}}(z),\sup\left\{ \tfrac{1}{4}(z\cdot z^{\prime})^{2}\mid z^{\prime}\in U\right\} \right)=\max\left(\iota_{S^{\perp}}(z),\tfrac{1}{4}\sigma_{U}^{2}(z)\right)\\
 & =\iota_{S^{\perp}}(z)+\tfrac{1}{4}\sigma_{U}^{2}(z)=\iota_{S^{\perp}}(z)+\tfrac{1}{4}\sigma_{C}^{2}(z).\end{align*}

The first part in (ii) is obvious, while the linearity of $\operatorname*{dom}\varphi_{D}$
and (iii) follow from the fact that $\sigma_{U}$ is an extended seminorm
(since $U=-U$).

(iv) Let us determine $\psi_{D}$. Because $\varphi_{D}=\tfrac{1}{4}\sigma_{C}^{2}+\iota_{S^{\perp}}$,
and $\sigma_{C}^{2}$, $\iota_{S^{\perp}}$ are proper ($w\times w^{\ast}-$)
lsc convex functions with $0\in\operatorname*{dom}\sigma_{C}^{2}\cap\operatorname*{dom}\iota_{S^{\perp}}$,
we have that $\psi_{D}=\varphi_{D}^{\square}=\operatorname*{cl}_{w\times w^{\ast}}[(\tfrac{1}{4}\sigma_{C}^{2})^{\ast}\square\iota_{S^{\perp}}^{\ast}]$.
Clearly, $\iota_{S^{\perp}}^{\ast}=\iota_{\operatorname*{cl}_{w\times w^{\ast}}(\operatorname*{aff}S)}$.

The crown $C$ is a symmetric set because so is $U$. Note that $\sigma_{U}^{\ast}=\sigma_{C}^{\ast}=\iota_{C}$
(see \cite[Th.\ 2.4.14]{Za-book}). Let us determine $(\tfrac{1}{4}\sigma_{C}^{2})^{\ast}$.
To do this we apply \cite[Th.\ 2.8.10 (iii)]{Za-book} for $X$ replaced
by $X\times X^{\ast}$, $Y=\mathbb{R}$, $Q=\mathbb{R}_{+}$, $f:=0$,
$g(t):=0$ for $t<0$, $g(t):=\tfrac{1}{4}t^{2}$ for $t\geq0$ and
$H:=\sigma_{C}$; clearly $g$ is $Q$--increasing, convex and continuous,
hence continuous at any $x_{0}\in\operatorname*{dom}f\cap H^{-1}(\operatorname*{dom}g)=\operatorname*{dom}\sigma_{C}$.
Moreover, $g^{\ast}(s)=\infty$ for $s<0$ and $g^{\ast}(s)=s^{2}$
for $s\geq0$. Applying \cite[Th.\ 2.8.10 (iii)]{Za-book} we obtain
that \begin{align*}
(\tfrac{1}{4}\sigma_{C}^{2})^{\ast}(z) & =\min\left\{ (s\sigma_{C})^{\ast}(z)+g^{\ast}(s)\mid s\geq0\right\} \\
 & =\min\left\{ (0\sigma_{C})^{\ast}(z),\inf\left\{ s(\sigma_{C})^{\ast}(s^{-1}z)+s^{2}\mid s>0\right\} \right\} \\
 & =\min\left\{ (\iota_{\operatorname*{dom}\sigma_{C}})^{\ast}(z),\inf\left\{ s\iota_{C}(s^{-1}z)+s^{2}\mid s>0\right\} \right\} .\end{align*}
 But \[
\inf\left\{ s\iota_{C}(s^{-1}z)+s^{2}\mid s>0\right\} =\inf\left\{ s^{2}\mid s>0,\ (z)\in sC\right\} =\left(p_{C}(z)\right)^{2},\]
 where $p_{C}$ is the Minkowski functional associated to $C$; $p_{C}$
is a seminorm on the linear space $\operatorname*{dom}p_{C}=\mathbb{R}_{+}C$.
Because for $f\in\Gamma(E)$ one has $f_{\infty}=\sigma_{\operatorname*{dom}f^{\ast}}$,
we have that $\operatorname*{cl}\operatorname*{dom}f^{\ast}=\partial f_{\infty}(0)$
(see e.g.\ \cite[Exer.\ 2.23]{Za-book}). Taking $f=\iota_{C}$,
we obtain that $\operatorname*{cl}(\operatorname*{dom}\sigma_{C})=\partial\iota_{C_{\infty}}(0)=(C_{\infty})^{0}$.
Hence $(\iota_{\operatorname*{dom}\sigma_{C}})^{\ast}=\iota_{(\operatorname*{dom}\sigma_{C})^{0}}=\iota_{C_{\infty}}$.
Since $0\in C$ we have that $C=C+C_{\infty}\supset C_{\infty}$,
and so $\left(p_{C}(z)\right)^{2}=0=(\iota_{\operatorname*{dom}\sigma_{C}})^{\ast}(z)$
for $(z)\in C_{\infty}$. Therefore, \[
(\tfrac{1}{4}\sigma_{C}^{2})^{\ast}(z)=\left(p_{C}(z)\right)^{2}\quad\forall z\in Z.\]

It follows that $\psi_{D}=(\operatorname*{cl}_{w\times w^{\ast}}\mu)^{2}$,
where $\mu:=p_{C}\square\iota_{L}$ with $L:=\operatorname*{cl}_{w\times w^{\ast}}(\operatorname*{lin}S)$.

(v) Because $p_{C}$ and $\iota_{L}$ are symmetric sublinear functionals,
so is $\mu$; hence $\operatorname*{cl}_{w\times w^{\ast}}\mu$ is
a symmetric sublinear functional, too. It follows that $\operatorname*{dom}\psi_{D}=\operatorname*{dom}\mu$
is a linear space. The mentioned inclusions follow from the fact that
$\psi_{D}={\operatorname*{cl}}_{w\times w^{\ast}}(\operatorname*{conv}c_{D})\leq c_{D}$,
and so\[
D\subset\lbrack\psi_{D}=c]\subset\operatorname*{dom}\psi_{D}=\operatorname*{dom}({\operatorname*{cl}}_{w\times w^{\ast}}c_{D})\subset{\operatorname*{cl}}_{w\times w^{\ast}}(\operatorname*{conv}(\operatorname*{dom}c_{D}))={\operatorname*{cl}}_{w\times w^{\ast}}(\operatorname*{conv}D).\]

(vi) Since $D$ is non-negative and $0\in D,$ clearly $D^{\perp}\subset\lbrack\varphi_{D}=0].$
Let $z\in\lbrack\varphi_{D}=0];$ then $z\cdot z^{\prime}\leq c(z^{\prime})$
for every $z^{\prime}\in D,$ whence $tz\cdot z^{\prime}=z\cdot(tz^{\prime})\leq c(tz^{\prime})=t^{2}c(z^{\prime})$
for all $t\in\mathbb{R}$ and $z^{\prime}\in D.$ Dividing by $t>0,$
then by $t<0,$ and letting $t\rightarrow0$ we get $z\cdot z^{\prime}=0$
for $z^{\prime}\in D,$ and so $z\in D^{\perp}.$ Hence $D^{\perp}=[\varphi_{D}=0].$
A simpler argument (using the symmetry of $C$) gives $[\sigma_{C}=0]=C^{\perp}.$
It is obvious that $P^{\perp}=U^{\perp}=C^{\perp}.$ Thus, from $D=S\cup P$
we get $D^{\perp}=S^{\perp}\cap P^{\perp}=S^{\perp}\cap C^{\perp}$.

(vii) We know that $D^{+}=[\varphi_{D}\leq c]=S^{\perp}\cap\lbrack\sigma_{C}^{2}\leq4c]$
is a non-negative double-cone. Hence $\operatorname*{Skew}([\sigma_{C}^{2}\leq4c])=[\sigma_{C}=0]\cap\lbrack c=0]=C^{\perp}\cap\lbrack c=0]$;
therefore $\operatorname*{Skew}(D^{+})=S^{\perp}\cap P^{\perp}\cap\lbrack c=0]=\operatorname*{Skew}(D^{\perp})=:D_{0}$
while $\operatorname*{Pos}([\sigma_{C}^{2}\leq4c])=\zeta^{-1}(2C^{\circ})$.
The proof is complete. \end{proof}

\strut

Let us call a double-cone $D$ \emph{positive} if $D$ is non-negative
and $\operatorname*{Skew}(D)=\{0\}$. In this case $\varphi_{D}=\frac{1}{4}\sigma_{C}^{2}$,
$\psi_{D}=p_{C}^{2}$, where $C=\operatorname*{Crown}(D)$, and $\operatorname*{Unit}(D^{+})=(2C^{\circ})\cap\mathbb{U}$,
$D^{+}=D_{0}\cup\zeta^{-1}(2C^{\circ})$.

Therefore a positive double-cone $D$ is monotone iff $D\subset D^{+}$
iff $C\subset2C^{\circ}$.

\begin{proposition} \label{dc mon}Let $D\subset Z$ be a non-negative
double-cone with skew part $S$, non-empty positive part $P$, unitary
part $U$, and crown $C$. TFAE:

\emph{(a)} $D$ is monotone,

\emph{(b)} $\left\langle z,z^{\prime}\right\rangle ^{2}\leq4c(z)\cdot c(z^{\prime})$
for all $z,z^{\prime}\in D$,

\emph{(c)} $S\subset D_{0}$ and $U\subset S^{\perp}\cap(2C^{\circ})$,

\emph{(d)} $C\subset S^{\perp}\cap(2C^{\circ})$ and $S\subset S^{\perp}$.

In this case $S$ is a skew monotone double-cone, $D\subset S^{\perp}$,
and $\operatorname*{cl}_{w\times w^{\ast}}(\operatorname*{conv}S)\subset D^{\perp}$.
\end{proposition}

\begin{proof} (a) $\Leftrightarrow$ (b) Assume that $D$ is monotone
and fix $z,z^{\prime}\in D$. Then $tz\in D$ for $t\in\mathbb{R}$,
and so \[
c(tz-z^{\prime})=t^{2}c(z)-t\left\langle z,z^{\prime}\right\rangle +c(z^{\prime})\geq0\quad\forall t\in\mathbb{R},\]
 which is equivalent to $c(z)$, $c(z^{\prime})\geq0$ and $\left\langle z,z^{\prime}\right\rangle ^{2}\leq4c(z)\cdot c(z^{\prime})$.

(a) $\Leftrightarrow$ (c) Both $D$ and $D^{+}$ are non-negative
double-cones with $\operatorname*{Skew}(D^{+})=D_{0}$ and $\operatorname*{Pos}(D^{+})=\zeta^{-1}(2C^{\circ})\cap S^{\perp}$.
Clearly, $D$ is monotone iff $D\subset D^{+}$ iff $S\subset\operatorname*{Skew}(D^{+})$
and $U\subset\operatorname*{Pos}(D^{+})$. Note that $U\subset\zeta^{-1}(2C^{\circ})\cap S^{\perp}$
iff $U\subset S^{\perp}\cap2C^{\circ}$.

(c) $\Leftrightarrow$ (d) The direct implication is obvious. For
the converse we have to prove that $S\subset D_{0}$, or equivalently
$S\subset D^{\perp}$ $(=S^{\perp}\cap P^{\perp})$. First observe
that $P^{\perp}=U^{\perp}=C^{\perp}$. Since $C\subset S^{\perp}$,
we have that $S\subset S^{\perp\perp}\subset C^{\perp}=P^{\perp}$.
Since by our hypothesis we have $S\subset S^{\perp}$, we get $S\subset D^{\perp}$.

The inclusion $D\subset S^{\perp}$ ($\Leftrightarrow S\subset D^{\perp}$)
follows from (b), while from $S\subset D^{\perp}$ we get $\operatorname*{cl}_{w\times w^{\ast}}(\operatorname*{conv}S)\subset D^{\perp}$.
\end{proof}

\begin{remark} \label{affmondc} It is interesting to observe that
if $D\subset Z$ is a monotone double-cone, $\operatorname*{aff}L$
could be non-monotone. For this claim consider $X:=\mathbb{R}^{2}$
identified with its dual, $z_{1}:=\left((0,1),(1,1)\right)$, $z_{2}:=\left((1,1),(1,0)\right)$,
$z_{3}:=\left((1,0),(1,0)\right)$ and $d_{i}:=\mathbb{R}z_{i}$.
We have that $c(z_{1})=c(z_{2})=c(z_{3})=1\geq0$; moreover $\left\langle z_{1},z_{2}\right\rangle =\left\langle z_{2},z_{3}\right\rangle =2$,
$\left\langle z_{1},z_{3}\right\rangle =1$, and so $\left\langle z_{1},z_{2}\right\rangle ^{2}\leq4c(z_{1})\cdot c(z_{2})$,
$\left\langle z_{2},z_{3}\right\rangle ^{2}\leq4c(z_{2})\cdot c(z_{3})$,
$\left\langle z_{1},z_{3}\right\rangle ^{2}\leq4c(z_{1})\cdot c(z_{3})$.
It follows that $D:=\cup_{i=1}^{3}\mathbb{R}z_{i}$ is a monotone
double-cone. However, $z:=z_{1}-2z_{2}+z_{3}=\left((-1,-1),(0,1)\right)\in\operatorname*{aff}D$
and $c(z)=-1<0$. Hence $\operatorname*{aff}L$ is not monotone. \end{remark}

It is clear that a linear subspace $L$ is monotone iff $L$ is non-negative
iff $L=S_{L}\cup P_{L}$.

\begin{theorem} \label{lin FP} Let $L\subset Z$ be linear monotone
subspace with skew part $S$, nonempty positive part $P$, unitary
part $U$, and crown $C$. Then

\emph{(i)} $S$ is a skew linear subspace of $Z$.

\emph{(ii)} $S+U=U$, $S+P=P$ and $S\subset\operatorname*{conv}U\subset C$;
in particular $S\subset C_{\infty}$, and so $\operatorname*{bar}U=\operatorname*{bar}C\subset S^{\perp}$.

\emph{(iii)} $\operatorname*{conv}U=\{z\in L\mid c(z)\le1\}=[c_{L}\le1]$.

\emph{(iv)} $\varphi_{L}=\frac{1}{4}\sigma_{C}^{2}$, $\operatorname*{dom}\varphi_{L}=\operatorname*{bar}C\subset S^{\perp}$
is a linear subspace of $Z$ and $\varphi_{L}(tz)=t^{2}\varphi_{L}(z)$,
for all $z\in Z$ and $t\in\mathbb{R}$.

\emph{(v)} $\psi_{L}=p_{C}^{2}$, $C=[\psi_{L}\leq1]$, $C_{\infty}=[\psi_{L}=0]$,
$\psi_{L}(tz)=t^{2}\psi_{L}(z)$ for all $z\in Z$ and $t\in\mathbb{R}$,
$\operatorname*{dom}\psi_{L}=\mathbb{R}C$ is a linear space, and
$L\subset\operatorname*{dom}\psi_{L}\subset\operatorname*{cl}_{w\times w^{\ast}}L$.

\emph{(vi)} $-C_{\infty}$ is a $w\times w^{*}$-closed monotone linear
subspace.

\emph{(vii)} If, moreover, $L$ is $w\times w^{*}$-closed then $C_{\infty}=S$;
in particular $C_{\infty}$ is skew and $S$ is $w\times w^{*}$-closed.
\end{theorem}

\begin{proof} (i) From Proposition \ref{dc mon} we have that $L\subset S^{\perp}$.
Consider $z,z^{\prime}\in S$; then $z+z^{\prime}\in L$ and $c(z+z^{\prime})=c(z)+z\cdot z^{\prime}+c(z^{\prime})=0$,
and so $z+z^{\prime}\in S$. Since $S$ is a double-cone, it follows
that $S$ is linear.

(ii) Because $0\in S$, clearly $U\subset S+U$. Let $z_{0}\in S$
and $z\in U$. Since $z\in L\subset S^{\perp}$, we have that $c(z_{0}+z)=c(z_{0})+z_{0}\cdot z+c(z)=c(z)=1$;
hence $z_{0}+z\in U$. It follows that $S+U\subset U$, whence $S+U=U$.
Similarly, $S+P=P$.

Fix $z_{1}\in U$; hence $-z_{1}\in U$. Then for $z_{0}\in S$ we
have that $z_{0}+z_{1},z_{0}-z_{1}\in U$, and so $z_{0}=\tfrac{1}{2}(z_{0}+z_{1})+\tfrac{1}{2}(z_{0}-z_{1})\in\operatorname*{conv}U$.
Hence $S\subset\operatorname*{conv}U\subset C$. Since $C$ is a closed
convex set containing $0$ we obtain that $S\subset C_{\infty}$.

Because for $f\in\Gamma(E)$ one has $f_{\infty}=\sigma_{\operatorname*{dom}f^{\ast}}$,
we have that $\operatorname*{cl}(\operatorname*{dom}f^{\ast})=\partial f_{\infty}(0)$
(see e.g.\ \cite[Exer.\ 2.23]{Za-book}). Taking $f=\iota_{C}$,
we obtain that $\operatorname*{cl}(\operatorname*{bar}C)=\operatorname*{cl}_{w\times w^{\ast}}(\operatorname*{dom}\sigma_{C})=\partial\iota_{C_{\infty}}(0)=(C_{\infty})^{\perp}$.
Hence $\operatorname*{bar}U=\operatorname*{bar}C\subset(C_{\infty})^{\perp}\subset S^{\perp}$.

(iii) Let $z\in L$ be such that $t:=c(z)\leq1$. If $t=0$ then $z\in S\subset\operatorname*{conv}U$.
If $t>0$ then $z^{\prime}:=t^{-1/2}z\in U\subset\operatorname*{conv}U$,
and so $z=(1-t^{1/2})0+t^{1/2}z^{\prime}\in\operatorname*{conv}U$.
Hence $[c_{L}\leq1]\subset\operatorname*{conv}U$. Since $c_{L}$
is convex and $U\subset\lbrack c_{L}\leq1]$ we also have $\operatorname*{conv}U\subset\lbrack c_{L}\leq1]$.

(iv) \&~(v) From (ii) we see that $\operatorname*{dom}\sigma_{C}\subset S^{\perp}$;
hence Proposition \ref{dc FP} gives $\varphi_{L}=\frac{1}{4}\sigma_{C}^{2}+\iota_{S^{\perp}}=\frac{1}{4}\sigma_{C}^{2}$,
and so $\psi_{L}=p_{C}^{2}$, $[\psi_{L}\le1]=[p_{C}\le1]=C$, $[\psi_{L}=0]=[p_{C}=0]=C_{\infty}$.
The remaining properties follow from Proposition \ref{dc FP}, too.

(vi) The set $C_{\infty}$ is linear since $C$ is symmetric. Also,
$C_{\infty}=[\psi_{L}=0]$ is dissipative, that is, $C_{\infty}\subset\lbrack c\leq0]$
since $\psi_{L}\geq c$; hence $-C_{\infty}$ is also monotone.

(vii) Assume that $L$ is $w\times w^{*}$-closed. Then $C_{\infty}\subset C\subset L$,
and so $C_{\infty}$ is monotone. Using (vi) we obtain that $C_{\infty}$
is skew, which implies $C_{\infty}\subset S$. From (ii) we get $C_{\infty}=S$.
\end{proof}

\strut

We complete Theorem \ref{lin FP} with the following result.

\begin{proposition} \label{lin psi} Let $L\subset Z$ be a linear
monotone subspace with skew part $S$, nonempty positive part $P$,
unitary part $U$, and crown $C$. Then

\emph{(i)} $[\psi_{L}=c]$ is linear and monotone, $C$ is not a linear
space and $C^{\perp}=L^{\perp}$; moreover, \[
\lbrack\psi_{L}=c]=\{z\mid c(z)\geq0,\ z\in\sqrt{c(z)}\cdot C\},\]
 where $0\cdot C:=C_{\infty}$, and \[
\lbrack\varphi_{L}=c]=\operatorname*{Skew}(L^{\perp})\cup\zeta^{-1}([\sigma_{C}=2])=\operatorname*{Skew}(C^{\perp})\cup\zeta^{-1}([\sigma_{C}=2]).\]

\emph{(ii)} $\operatorname*{Skew}([\psi_{L}=c])=\operatorname*{Skew}(C_{\infty})$
and $\operatorname*{Skew}([\psi_{L}=c])$ is a $w\times w^{\ast}-$closed
(skew) and linear subspace, $\operatorname*{Pos}([\psi_{L}=c])=\zeta^{-1}(C)$,
$\operatorname*{Unit}([\psi_{L}=c])=\operatorname*{Unit}(C)$, and
$\operatorname*{Crown}([\psi_{L}=c])=C$;

\emph{(iii)} $[\psi_{L}=c]\subset\operatorname*{dom}\psi_{L}=\mathbb{R}C\subset\operatorname*{cl}_{w\times w^{*}}L\subset S_{[\psi_{L}=c]}^{\perp}$,
$\varphi_{[\psi_{L}=c]}=\frac{1}{4}\sigma_{C}^{2}$, $\psi_{[\psi_{L}=c]}=p_{C}^{2}$.
\end{proposition}

\begin{proof} (i) To prove that $[\psi_{L}=c]$ is linear it suffices
to show that $z+z^{\prime}\in\lbrack\psi_{L}=c]$, for every $z,z^{\prime}\in\lbrack\psi_{L}=c]$.
To this end fix $z^{\prime}\in L$. Then the function $\eta_{z^{\prime}}:Z\rightarrow\mathbb{R}\cup\{\infty\}$
defined by\[
\eta_{z^{\prime}}(z)=\psi_{L}(z+z^{\prime})-z\cdot z^{\prime}-c(z^{\prime})\quad(z\in Z),\]
 is proper convex and lsc. For $z\in L$ we have $z+z^{\prime}\in L$,
and so $\psi_{L}(z+z^{\prime})=c(z+z^{\prime})=c(z)+z\cdot z^{\prime}+c(z^{\prime})$.
It follows that $\eta_{z^{\prime}}\leq c_{L}$. This yields\[
\psi_{L}(z+z^{\prime})\leq\psi_{L}(z)+z\cdot z^{\prime}+c(z^{\prime})\quad\forall z\in Z,\ \forall z^{\prime}\in L.\]
 Hence if $z\in\lbrack\psi_{L}=c]$ then $z+z^{\prime}\in\lbrack\psi_{L}\leq c]=[\psi_{L}=c]$.
We proved $z+z^{\prime}\in\lbrack\psi_{L}=c]$ for all $z\in\lbrack\psi_{L}=c]$
and $z^{\prime}\in L$. Now consider the function $\eta_{z^{\prime}}$
from above with $z^{\prime}\in\lbrack\psi_{L}=c]$ and repeat the
argument to conclude that $[\psi_{L}=c]$ is linear.

The formula for $[\psi_{L}=c]$ follows from $\psi_{L}=p_{C}^{2}$
and Theorem \ref{lin FP} (v). Since $\varphi_{L}=\tfrac{1}{4}\sigma_{C}^{2}\geq0$
(and using Proposition \ref{dc FP} (vi)), we have that \[
\lbrack\varphi_{L}=c]=\left([\sigma_{C}=0]\cap\lbrack c=0]\right)\cup\left(\lbrack\sigma_{C}^{2}=4c]\cap\lbrack c>0]\right)=\operatorname*{Skew}(C^{\perp})\cup\zeta^{-1}([\sigma_{C}=2]).\]

By Theorem \ref{lin FP} (ii) we have $S\subset C,$ whence $C^{\perp}\subset S^{\perp};$
using Proposition \ref{dc FP} (vi) we get $L^{\perp}=C^{\perp}\cap S^{\perp}=C^{\perp}.$

Assume that $C$ is linear. Then $C=C_{\infty}=[\psi_{L}=0].$ It
follows that $c(z)\leq\psi_{L}(z)=0$ for every $z\in C,$ and so
we get the contradiction $1=c(z)\leq0$ for every $z\in U$ $(\subset C).$

(ii) From Theorem \ref{lin FP} (v) we get \[
\operatorname*{Skew}([\psi_{L}=c])=[\psi_{L}=c]\cap\lbrack c=0]=[\psi_{L}=0]\cap\lbrack c=0]=C_{\infty}\cap\lbrack c=0]=\operatorname*{Skew}(C_{\infty}).\]
 One obtains similarly the other equalities. Because $-C_{\infty}$
is a $w\times w^{\ast}$--closed monotone linear subspace, using Theorem
\ref{lin FP} (vii) we obtain that $\operatorname*{Skew}(C_{\infty})=-\operatorname*{Skew}(-C_{\infty})$
is $w\times w^{\ast}$--closed and linear.

(iii) All the facts are consequences of Theorem \ref{lin FP}, Proposition
\ref{dc mon}, (i), and Proposition \ref{caracter-NI} (i). \end{proof}

\strut

In the sequel we discuss the representability of a linear subspace.
Note that the smallest representable operator that contains a monotone
double-cone $D$, namely $[\psi_{D}=c]$, is a (monotone) double-cone.
In this case it makes sense to talk about the skew, positive, and
unitary parts as well as the crown of $[\psi_{D}=c]$.

\begin{proposition} \label{linrep skewclosed}If $L\subset Z$ is
linear and representable then $S_{L}$ is a ${w\times w^{*}}$-closed
skew linear subspace. \end{proposition}

\begin{proof} We know that $S_{L}$ is a skew linear subspace from
Theorem \ref{lin FP}. Hence, by Proposition \ref{skewdc-char}, $\operatorname*{cl}_{w\times w^{\ast}}S_{L}$
is skew and $\iota_{\operatorname*{cl}_{w\times w^{\ast}}S_{L}}=\psi_{S_{L}}\geq\psi_{L}\geq c$
because $L$ is monotone. This yields $\operatorname*{cl}_{w\times w^{\ast}}S_{L}\subset\lbrack\psi_{L}=c]\cap\lbrack c=0]=L\cap\lbrack c=0]=S_{L}$.
\end{proof}

\strut

Let us define the \emph{$w\times w^{\ast}-$natural convergence} ($\nu$
for short) of nets in $Z$ as \[
z_{i}\rightarrow^{\nu}z\Longleftrightarrow[c(z_{i})\leq c(z)\ \forall i\in I,\ \mathrm{and}\ z_{i}\rightarrow z,\ w\times w^{\ast}\ \mathrm{in}\ Z],\]
 and the \emph{$w\times w^{\ast}-$}natural closure of a subset $A\subset Z$,
denoted by $\operatorname*{cl}_{\nu}A:=\{z\in Z\mid\exists\ (z_{i})_{i}\subset A:z_{i}\rightarrow^{\nu}z\}$.
A set $A\subset Z$ is called $\nu-$\emph{closed} if $A=\operatorname*{cl}_{\nu}A$.
Similarly, a function $f:Z\rightarrow\overline{\mathbb{R}}$ is $\nu-$lsc
if $f(z)\leq\lim_{i}f(z_{i})$ whenever $z_{i}\rightarrow^{\nu}z$
($z_{i},z\in Z$).

Observe that the $w\times w^{*}-$natural convergence introduced above
is different of the convergence induced by the natural topology defined
by Penot in \cite{Penot:08}.

Note that for every $T\in\mathcal{M}(X)$, $[\psi_{T}=c]$ is $\nu$--closed;
in particular every representable operator is $\nu$--closed. Indeed,
let $(z_{i})_{i\in I}\subset\lbrack\psi_{T}=c]$ such that $z_{i}\rightarrow^{\nu}z\in Z$.
Then\[
\psi_{T}(z)\leq\liminf_{i\in I}\psi_{T}(z_{i})\leq\limsup_{i\in I}\psi_{T}(z_{i})=\limsup_{i\in I}c_{T}(z_{i})\leq c(z),\]
 which proves that $z\in\lbrack\psi_{T}=c]$. Moreover $\psi_{T}(z)=\lim_{i}\psi_{T}(z_{i})$
and similarly $f(z)=\lim_{i}f(z_{i})$ for every $f\in\mathcal{R}_{T}$
and $(z_{i})_{i\in I}\subset\lbrack f=c]$ with $z_{i}\rightarrow^{\nu}z$.
The following theorem shows that a monotone linear subspace $L$ of
$Z$ is representable iff $L$ is $\nu-$closed.

\begin{theorem} \label{lin repres}Let $L\subset Z$ be a linear
monotone subspace with nonempty positive part. TFAE:

\emph{(i)} $L$ is representable,

\emph{(ii)} $\operatorname*{Unit}([\psi_{L}=c])\subset L$,

\emph{(iii)} $c_{L}$ is $\nu$--lsc,

\emph{(iv)} $L$ is $\nu$--closed. \end{theorem}

\begin{proof} The implication (i) $\Rightarrow$ (ii) is straightforward.

(ii) $\Rightarrow$ (i) Let $z\in\lbrack\psi_{L}=c]$. If $c(z)>0$
then $\frac{1}{\sqrt{c(z)}}z\in U_{[\psi_{L}=c]}\subset L$ and so
$z\in L$. Since $[\psi_{L}=c]$ is linear (by Proposition \ref{lin psi}
(i)) with nonempty positive part, if $c(z)=0$, by Theorem \ref{lin FP}
(ii), $z\in\operatorname*{Skew}([\psi_{L}=c])\subset\operatorname*{conv}(\operatorname*{Unit}([\psi_{L}=c]))\subset L$.
Hence $L=[\psi_{L}=c]$.

(i) $\Rightarrow$ (iii) Assume that $L$ is representable. To prove
(iii) it is sufficient to consider the net $(z_{i})_{i\in I}\subset L$
with $z_{i}\to^{\nu}z$, that is, $z_{i}\rightarrow z$ for the topology
$w\times w^{\ast}$ and $c(z_{i})\leq c(z)$ for every $i\in I$.
Note that\[
c(z)\leq\psi_{L}(z)\leq\liminf_{i\in I}\psi_{L}(z_{i})=\liminf_{i\in I}c(z_{i})\leq c(z),\]
 which yields that $z\in\lbrack\psi_{L}=c]=L$ and so $c_{L}$ is
$\nu$--lsc.

The implication (iii) $\Rightarrow$ (iv) is plain.

(iv) $\Rightarrow$ (ii) Let $z\in\operatorname*{Unit}([\psi_{L}=c])$,
that is, $\psi_{L}(z)=c(z)=1$. Using Proposition \ref{lin psi} (ii)
and Theorem \ref{lin FP} (iii) we get $z\in C=\operatorname*{cl}_{w\times w^{\ast}}[c_{L}\leq1]$,
and so there exists $(z_{i})_{i\in I}\subset\lbrack c_{L}\leq1]$
$(\subset L)$ with $z_{i}\rightarrow z$ for the topology $w\times w^{\ast}$
in $Z$. Since $c(z)=1$ and $L$ is $\nu$--closed we get $z\in L$.
\end{proof}

\begin{remark} Note that whenever $S$ is a skew linear subspace
of $Z$, $\operatorname*{cl}_{\nu}S=\operatorname*{cl}_{w\times w^{*}}S$
since $\operatorname*{cl}_{\nu}S\subset\operatorname*{cl}_{w\times w^{*}}S$
and $\operatorname*{cl}_{w\times w^{*}}S$ is skew. In this case $S$
is representable iff $S$ is $w\times w^{*}-$closed, that is, we
recover part of Corollary \ref{skewmondc}. \end{remark}

\begin{corollary} If $L\subset Z$ is a linear monotone and $w\times w^{\ast}$--closed
subspace then $L$ is representable. \end{corollary}

\begin{corollary} If $L\subset Z$ is a linear representable subspace
and $(z_{i})_{i\in I}\subset L$ is such that $z_{i}\rightarrow^{\nu}z$,
then $c(z)=\lim_{i}c(z_{i})$, and $f(z)=\lim_{i}f(z_{i})$ for every
$f\in\mathcal{R}_{L}$. \end{corollary}

In the next result we characterize the uniqueness of monotone double-cones.

\begin{proposition} \label{dc unic} Let $D\subset Z$ be a monotone
double-cone. Then $D$ is unique iff $[\varphi_{D}=c]$ is monotone.
\end{proposition}

\begin{proof} According to Proposition \ref{caracter-unic} it is
clear that $D$ unique implies that $D_{1}:=[\varphi_{D}=c]$ is monotone.
For the converse it suffices to prove that $D$ non-unique implies
that $D_{1}$ is non-monotone. Since $D$ is not unique we can find
$z_{1},z_{2}\in[\varphi_{D}\le c]$ such that $c(z_{1}-z_{2})<0$.
Let $d:=\{tz_{1}+(1-t)z_{2}\mid t\in\mathbb{R}\}$ be the line through
$z_{1},z_{2}$; because $z_{1},z_{2}\in\operatorname*{dom}\varphi_{D}$
and $\operatorname*{dom}\varphi_{D}$ is a linear space (see Proposition
\ref{dc FP}), we have that $d\subset\operatorname*{dom}\varphi_{D}$.
Because $c(z_{1}-z_{2})<0$ and $\varphi_{D}$ is convex, the function
$\mathbb{R}\ni t\mapsto\eta(t):=(\varphi_{D}-c)(tz_{1}+(1-t)z_{2})$
is finite-valued, continuous and coercive, i.e., $\lim_{|t|\rightarrow+\infty}\eta(t)=+\infty$.
Since $\eta(0)\le0$ and $\eta(1)\le0$, there exist $t_{0}\le0$
and $t_{1}\ge1$ such that $\eta(t_{0})=\eta(t_{1})=0$. Then $z_{1}^{\prime}:=t_{0}z_{1}+(1-t_{0})z_{2}\in D_{1}$
and $z_{2}^{\prime}:=t_{1}z_{1}+(1-t_{1})z_{2}\in D_{1}$. Because
$c(z_{1}^{\prime}-z_{2}^{\prime})=(t_{1}-t_{0})^{2}c(z_{2}-z_{1})<0$,
we have that $D_{1}$ is not monotone. \end{proof}

\begin{corollary} \label{lin unic} Let $L\subset Z$ be linear monotone
with non-empty crown $C$. TFAE: \emph{(a)} $L$ is unique, \emph{(b)}
$[\varphi_{L}=c]$ is monotone; $\emph{(c)}$ $\zeta^{-1}(2C^{\circ})\cup\operatorname*{Skew}(L^{\perp})$
is monotone; \emph{(d)} $\zeta^{-1}([\sigma_{C}=2])\cup\operatorname*{Skew}(L^{\perp})$
is monotone. \end{corollary}

From Propositions \ref{NI} and \ref{dc unic} we get immediately
the next result.

\begin{corollary} \label{dc NI} Let $D\subset Z$ be a monotone
double-cone. Then $D$ is $NI$ iff either $[\varphi_{D}=c]$ is monotone
and $\operatorname*{dom}\varphi_{D}$ is not monotone or $\operatorname*{dom}\varphi_{D}=[\varphi_{D}=c]$.
\end{corollary}

This result together with Proposition \ref{lin psi} (i) and Remark
\ref{rem5} yield the next result. \begin{corollary} \label{lin NI}
Let $L\subset Z$ be linear monotone with non-empty crown $C$ and
set $M:=[\varphi_{L}=c]=\zeta^{-1}([\sigma_{C}=2])\cup\operatorname*{Skew}(L^{\perp}).$
Then $L$ is NI iff either $M$ is monotone and $\operatorname*{bar}C$
is not monotone or $\varphi_{L}=c_{M}$. \end{corollary}

\begin{remark} When $S$ is a skew linear subspace of $Z$ Corollary
\ref{dc NI} spells $S$ is NI iff either $S_{0}:=S^{\perp}\cap[c=0]$
is monotone and $S^{\perp}$ is not monotone or $S_{0}=S^{\perp}$
(that is $S^{\perp}$ is skew). According to Corollary \ref{skew skew}
and Proposition \ref{skew unic}, this comes to $S$ is NI iff $-S^{\perp}$
is monotone, and we recover part of Proposition \ref{skew NI}. \end{remark}

{The} following chart comprises the information on a linear non-skew
monotone subspace $L$ with crown $C$.

\begin{center}
\begin{tabular}{l}
1. $L$ is representable iff $L$ is $\nu-$closed\tabularnewline 2.
$L$ is unique and not NI iff $\operatorname*{bar}C$ is monotone and
$\zeta^{-1}([\sigma_{C}=2])\cup\operatorname*{Skew}(L^{\perp})\subsetneq\operatorname*{bar}C$\tabularnewline
3. $L$ is unique iff
$\zeta^{-1}([\sigma_{C}=2])\cup\operatorname*{Skew}(L^{\perp})$ is
monotone\tabularnewline 4. $L$ is NI iff either
$\zeta^{-1}([\sigma_{C}=2])\cup\operatorname*{Skew}(L^{\perp})$ is
monotone and $\operatorname*{bar}C$ is not monotone\tabularnewline
\hspace*{4mm}  or $\varphi_{L}=c_{[\varphi_{L}=c]}$\tabularnewline
5. $L$ is maximal monotone iff $L$ is $\nu-$closed and
NI\tabularnewline
\end{tabular}
\par\end{center}

\eject

\end{document}